\documentclass{alggeom}
\usepackage{mathtools, microtype, mathrsfs, xfrac, amssymb, enumerate, xspace, stmaryrd, amsthm}
\usepackage[alphabetic]{amsrefs}
\usepackage{ascii}
\usepackage[latin1]{inputenc}
\usepackage{tikz-cd}
\usepackage[all,cmtip]{xy}
\usepackage{hyperref}
\hypersetup{
	colorlinks=true,
	linkcolor=blue,
	filecolor=magenta,      
	urlcolor=cyan,
}
\usepackage{faktor}
\usepackage{cleveref}
\usepackage{comment}
\usepackage{graphicx}
\usepackage{aliascnt}
\usepackage{float}

\numberwithin{equation}{section}

\numberwithin{subsection}{section}

\allowdisplaybreaks[1]

\newtheorem{theorem}{Theorem}[section]
\newtheorem{proposition}[theorem]{Proposition}
\newtheorem{proposition-definition}[theorem]
{Proposition-Definition}
\newtheorem{corollary}[theorem]{Corollary}
\newtheorem{lemma}[theorem]{Lemma}

\theoremstyle{definition}
\newtheorem{definition}[theorem]{Definition}

\newtheorem{remark}[theorem]{Remark}

\theoremstyle{remark}

\newcommand\cA{\mathcal{A}} \newcommand\cB{\mathcal{B}}
\newcommand\cC{\mathcal{C}} \newcommand\cD{\mathcal{D}}
\newcommand\cE{\mathcal{E}} \newcommand\cF{\mathcal{F}}
 \newcommand\cH{\mathcal{H}}

\newcommand\cM{\mathcal{M}} 
\newcommand\cO{\mathcal{O}} 
 
\newcommand\cS{\mathcal{S}} 
 
 \newcommand\cX{\mathcal{X}}
\newcommand\cY{\mathcal{Y}} 

\renewcommand\AA{\mathbb{A}} 

\newcommand\GG{\mathbb{G}} \newcommand\HH{\mathbb{H}}
 
 \newcommand\LL{\mathbb{L}}
 \newcommand\NN{\mathbb{N}}
 \newcommand\PP{\mathbb{P}}

 \newcommand\ZZ{\mathbb{Z}}

\newcommand\rma{\mathrm{a}}

\newcommand\rmm{\mathrm{m}}

\newcommand\arr{\ifinner\to\else\longrightarrow\fi}

\newcommand\arrto{\ifinner\mapsto\else\longmapsto\fi}

\renewcommand\H{\operatorname{H}}

\newcommand\into{\hookrightarrow}

\def\displaytimes_#1{\mathrel{\mathop{\times}\limits_{#1}}}

\def\displayotimes_#1{\mathrel{\mathop{\bigotimes}\limits_{#1}}}

\renewcommand\hom{\operatorname{Hom}}

\newcommand\curshom{\mathop{\mathcal{H}\hskip -.8pt om}\nolimits}

\newcommand\ext{\operatorname{Ext}}

\newcommand\tor{\operatorname{Tor}}

\newcommand\pic{\operatorname{Pic}}

\newcommand\spec{\operatorname{Spec}}

\newcommand{\proj}{\operatorname{Proj}}

\newcommand\id{\mathrm{id}}

\newcommand{\underhom}{\mathop{\underline{\mathrm{Hom}}}\nolimits}

\newcommand{\underaut}{\mathop{\underline{\mathrm{Aut}}}\nolimits}

\newlength{\ignora}

\renewcommand{\setminus}{\smallsetminus}

\newcommand{\gm}{\GG_{\rmm}}

\newcommand{\GL}{\mathrm{GL}}

\newcommand{\ga}{\GG_{\rma}}

\newcommand{\ds}[1]{[\mspace{-2mu}[#1]\mspace{-2mu}]}

\DeclareFontFamily{U}{mathx}{\hyphenchar\font45}
\DeclareFontShape{U}{mathx}{m}{n}{
	<5> <6> <7> <8> <9> <10>
	<10.95> <12> <14.4> <17.28> <20.74> <24.88>
	mathx10
}{}
\DeclareSymbolFont{mathx}{U}{mathx}{m}{n}
\DeclareFontSubstitution{U}{mathx}{m}{n}
\DeclareMathAccent{\widecheck}{0}{mathx}{"71}
\DeclareMathAccent{\wideparen}{0}{mathx}{"75}

\renewcommand{\epsilon}{\varepsilon}

\newcommand{\Mbar}{\overline{\cM}}

\newcommand{\Mtilde}{\widetilde{\mathcal M}}

\newcommand{\Ctilde}{\widetilde{\mathcal C}}

\newcommand{\mt}{\widetilde{\mathcal M}}

\begin{document}

\title{The moduli stack of $A_r$-stable curves}
\author{Michele Pernice}
\email{m.pernice(at)kth.se}
\address{KTH, Room 1642, Lindstedtsv\"agen 25,
114 28, Stockholm }



\begin{abstract}
	
	This paper is the first in a series of four papers aiming to describe the (almost integral) Chow ring of $\Mbar_3$, the moduli stack of stable curves of genus $3$. In this paper, we introduce the moduli stack $\Mtilde_{g,n}^r$ of $n$-pointed $A_r$-stable curves and extend some classical results about $\Mbar_{g,n}$ to $\Mtilde_{g,n}^r$, namely the existence of the contraction morphism. Moreover, we describe the normalization of the locally closed substack of $\Mtilde_{g,n}^r$ parametrizing curves with $A_h$-singularities for a fixed $h\leq r$. 
\end{abstract}

\maketitle
\section*{Introduction}

The geometry of the moduli spaces of curves has always been the subject of intensive investigations, because of its manifold implications, for instance in the study of families of curves. One of the main aspects of this investigation is the intersection theory of these spaces, which can used to solve either geometric, enumerative or arithmetic problems regarding families of curves. In his groundbreaking paper \cite{Mum}, Mumford introduced the intersection theory with rational coefficients for the moduli spaces of stable curves. After almost two decades, Edidin and Graham introduced in \cite{EdGra} the intersection theory of global quotient stacks (therefore in particular for moduli stacks of stable curves) with integer coefficients.

To date, several computations have been carried out. While the rational Chow ring of $\cM_g$, the moduli space of smooth curves, is known for $2\leq g\leq 9$ (\cite{Mum}, \cite{Fab}, \cite{Iza}, \cite{PenVak}, \cite{CanLar}), the complete description of the rational Chow ring of $\Mbar_g$, the moduli space of stable curves, has been obtained only for genus $2$ by Mumford and for genus $3$ by Faber in \cite{Fab}. As expected, the integral Chow ring is even harder to compute: the only complete description of the integral Chow ring of the moduli stack of stable curves is the case of $\Mbar_2$, obtained by Larson in \cite{Lar} and subsequently with a different strategy by Di Lorenzo and Vistoli in \cite{DiLorVis}. It is also worth mentioning the result of Di Lorenzo, Pernice and Vistoli regarding the integral Chow ring of $\Mbar_{2,1}$, see \cite{DiLorPerVis}.

The aim of this series of four papers is to describe the Chow ring with $\ZZ[1/6]$-coefficients of the moduli stack $\Mbar_3$ of stable genus $3$ curves. This provides a refinement of the result of Faber with a completely indipendent method, which has the potential to give a more algorithmic way to compute these Chow rings. The approach is a generalization of the one used in \cite{DiLorPerVis}: we introduce an Artin stack, which is called the stack of $A_r$-stable curves, where we allow curves with $A_r$-singularities to appear. The idea is to compute the Chow ring of this newly introduced stack in the genus $3$ case and then, using localization sequence, find a description for the Chow ring of $\Mbar_3$. The stack $\Mtilde_{g,n}$ introduced in \cite{DiLorPerVis} is cointained as an open substack inside our stack. 

\subsection*{Outline of the paper}

This is the first paper in the series. It focuses on introducing the moduli stack $\Mtilde_g^r$ of $A_r$-stable curves: this is the new actor for the computation of the Chow ring of $\Mbar_3$. Firstly, we prove several results which are straightforward generalizations of the stable case. We construct the contraction morphism for the $A_r$-stable case, which is fundamental for some of the computations of the Chow ring of $\Mtilde_3^7$. Finally, we describe the locally closed substack parametrizing curves with $A_n$-singularities for a fixed $n$. The description will be fundamental in the fourth paper to understand the ideal of relations in the Chow ring of $\Mtilde_3^7$ coming from the closed substack $\Mtilde_3^7\setminus \Mbar_3$ and, using localization sequence, to eventually get the Chow ring of $\Mbar_3$ as a quotient of the Chow ring of $\Mtilde_3^7$. 

Specifically, \Cref{sec:1} is dedicated to introducing $A_r$-singularity (and $A_r$-stable curves) and studying contractions in the setting of $A_r$-stable curves. In particular, we extend some of the results cointained in the first two chapter of \cite{Knu} to our setting, simplifying the proofs where possible. The main result in this section is the existence and uniqueness of (minimal) contractions.

In \Cref{sec:2}, we define the moduli stack $\Mtilde_{g,n}^r$ of $n$-pointed $A_r$-stable curves of genus $g$ and prove that it is a smooth global quotient stack. We also prove that it shares some features with the classic stable case, for example the existence of the Hodge bundle. The main theorem of this section is that the contraction morphism is an open immersion (instead of an isomorphism as in the stable case).

\Cref{sec:3} focuses on proving some preliminary results needed for \Cref{sec:4}. It is divided in two indipendent parts. In the first one, we describe a variant of the stack of finite flat algebras, namely the stack of finite flat extensions of degree $d$ of finite flat (curvilinear) algebras and prove that it is an affine bundle over a classifying stack. In the second one, we prove some technical results regarding blowups and pushouts that are probably well-known to experts, aiming to verify that the two constructions behave well in families and are inverse of each other in our setting.

Finally, in \Cref{sec:4} we describe the closed complement of $\Mbar_{g,n}$ inside $\Mtilde_{g,n}^r$, stratifying it by complexity of the singularities that can appear in the curves. We use pushouts and blowups to describe the locally closed substack parametrizing curves with a $A_h$-singularity and we prove that it is (up to an explicit finite \'etale cover) an affine bundle over some moduli stack of curves (which is also described completely) with stricly lower genus.

\section{$A_r$-stable curves and contractions}\label{sec:1}
Fix a nonnegative integer $r$.

\begin{definition}	Let $k$ be an algebraically closed field and $C/k$ be a proper reduced connected one-dimensional scheme over $k$. We say the $C$ is an \emph{$A_r$-prestable curve} if it has at most $A_r$-singularity, i.e. for every closed point $p\in C(k)$, we have an isomorphism
$$ \widehat{\cO}_{C,p} \simeq k[[x,y]]/(y^2-x^{h+1}) $$ 
with $ 0\leq h\leq r$. 
Moreover, suppose $\Sigma \subset C$ is an effective Cartier divisor whose support is in the smooth locus of a $A_r$-prestable curve $C$. We say that the pair $(C,\Sigma)$ is $A_r$-semistable if the following two conditions are verified:
\begin{itemize}
 \item for every irreducible component $\Gamma$ of $C$ the integer $a_{\Gamma}:=\deg \omega_C(\Sigma)\vert_{\Gamma} $ is nonnegative;
 \item there exists at most one irreducible component $\Gamma_0$ such that $a_{\Gamma_0} =0$; such component is called the semistable component. 
\end{itemize}
 we say that a pair is $A_r$-stable if $a_{\Gamma}>0$ for every irreducible component $\Gamma$ of $C$.

 Finally, a $n$-pointed $A_r$-stable curve over $k$ is $A_r$-prestable curve together with $n$ smooth distinct closed points $p_1,\dots,p_n$ such that the pair $(C,\cO_C(p_1+\dots+p_n))$ is $A_r$-stable, i.e. $\omega_C(p_1+\dots+p_n)$ is ample. If $n=0$, we just say that $C$ is a $A_r$-stable curve.
\end{definition}
\begin{remark}
	Notice that a $A_r$-prestable curve is l.c.i by definition, therefore the dualizing complex is in fact a line bundle. 
\end{remark}

Notice that every time we talk about genus, we intend arithmetic genus, unless specified otherwise.

\begin{remark}\label{rem:genus-count}
	Let $C$ be a connected, reduced, one-dimensional, proper scheme over an algebraically closed field. Let $p$ be a rational point of $C$ which is a $A_r$-singularity. We denote by $b:\widetilde{C}\arr C$ the partial normalization at the point $p$ and by $J_b$ the conductor ideal of $b$. Then a straightforward computation shows that \begin{enumerate}
		\item if $r=2h$, then $g(C)=g(\widetilde{C})+h$;
		\item if $r=2h+1$ and $\widetilde{C}$ is connected, then $g(C)=g(\widetilde{C})+h+1$,
		\item if $r=2h+1$ and $\widetilde{C}$ is not connected, then $g(C)=g(\widetilde{C})+h$.
	\end{enumerate}
	If $\widetilde{C}$ is not connected, we say that $p$ is a separating point. Furthermore, Noether formula gives us that $b^*\omega_C \simeq \omega_{\widetilde{C}}(J_b^{\vee})$.
\end{remark}

The rest of this section is dedicated to extend to theory of contractions of $n$-pointed ($A_1$-)stable curves to the case of $A_r$-stable curves. More generally, we will prove several results about contractions of $A_r$-semistable pairs over algebraically closed fields, namely the existence and unicity of contractions.

The following results are the natural generalizations of the ones in the first section of \cite{Knu}, where the author proves them in the context of classical stable curves. We follow Knutsen's strategy, simplifying where possible. 

\begin{definition}
 Let $(C,\Sigma)$ be a $A_r$-semistable pair. Then a contraction of such a pair is a morphism  $c:C \rightarrow C'$ of $A_r$-prestable curves such that 
 \begin{itemize}
  \item[(i)] if $a_{\Gamma}>0$ for every irreducible component $\Gamma$ of $C$, then $f$ is an isomorphism;
  \item[(ii)] if $a_{\Gamma}=0$ for some irreducible component $\Gamma_0$ then:
  \begin{itemize}
     \item[a)] $c(\Gamma_0)$ is contracted to a closed point $p_{0}$ and $c^{-1}(p_{0})\simeq \Gamma_0$ (schematically); 
     \item[b)]  the restriction
  $$ c\vert_{C\setminus \Gamma_0}: C\setminus \Gamma_0 \longrightarrow C' \setminus p_0$$ 
  is an isomorphism.
 \end{itemize}
  
 \end{itemize}

\end{definition}

\begin{remark} \label{rem:poss-contrac}
 It is important to remark that if $(C,\Sigma)$ is a $A_r$-semistable pair and $\Gamma$ is an irreducible component of $C$, then $a_{\Gamma}=0$ implies that $g(\Gamma)=0$ and therefore $\Gamma$ is a projective line. Moreover, we can have only three possible scenarios: 
 \begin{itemize}
  \item[(1)] $\Gamma$ intersects the rest of the curve in one node and $\Sigma\cap \Gamma$ is the divisor associated to a point in $\Gamma$ (different from the node);
  \item[(2)] $\Gamma$ intersects the rest of the curve in two nodes and $\Sigma\cap \Gamma = \emptyset$;
  \item[(3)] $\Gamma$ interescts the rest of the curve in one tacnode and $\Sigma\cap \Gamma = \emptyset$.
 \end{itemize}
The cases (1) and (2) come from the classical stable case and they were deeply studied by Knutsen in \cite{Knu}; these components are denoted respectively \emph{rational tails} and \emph{rational bridges}.  The only new components are of type (3), which is the limit of case (2) when the two nodes coincide, and we will denote such components as \emph{rational almost-bridges}. 
\end{remark}

\begin{remark}\label{rem:cart-div}
 Using the previous remark, one can see that if $c:C\rightarrow C'$ is a contraction of a $A_r$-semistable pair $(C,\Sigma)$, then $c(\Sigma)\subset C'$ is still an effective Cartier divisor and the morphism $c\vert_{\Sigma}: \Sigma \rightarrow c(\Sigma)$ is an isomorphism. Moreover, it follows easily that $(C',c(\Sigma))$ is a $A_r$-stable pair.
\end{remark}

First of all, we want to prove the existence of such contractions. To do so, we need a technical result, which will be very useful later on. This result is a generalization of Theorem 1.8 of \cite{Knu} for $A_r$-semistable pairs.

\begin{proposition}\label{prop:boundedness}
	Let $(C,\Sigma)$ be a $A_r$-semistable pair over an algebraically closed field $k$. Then 
	\begin{enumerate}
		\item[i)] $\H^1(C,\omega_C(\Sigma)^{\otimes m})=0$ for every $m\geq 2$,
		\item[ii)] $\omega_C(\Sigma)^{\otimes m}$ is globally generated for every $m\geq 2$.
		\item[iii)] $\omega_C(\Sigma)^{\otimes m}$ is normally generated for $m\geq 4$.
	\end{enumerate}
\end{proposition}

\begin{proof}
	 Using duality, i) is equivalent to 
	$$ \H^0(C,\omega_C^{\otimes (1-m)}(-m\Sigma))=\H^0(C,\omega_C(\Sigma)^{\otimes (1-m)}(-\Sigma))=0$$
	for every $m\geq 2$. As before, we denote by $a_{\Gamma}$ the degree of $\omega_C(\Sigma)\vert_{\Gamma}$ where $\Gamma$ is an irreducible component of $C$. Thus, 
	$$\deg \big(\omega_C(\Sigma)^{\otimes 1-m}(-\Sigma)\big) \leq (1-m)a_{\Gamma}\leq 0$$
	for every $m\geq2$. Thus the sections of the line bundle are zero restricted on every irreducible component $\Gamma$ such that $a_{\Gamma}>0$. Because the restriction of the line bundle to a semistable component is non-canonically isomorphic to $\cO_{\PP^1}$, which is globally generated, it easily follows that all the sections must vanish.

	Regarding ii), we need to prove that for every closed point $p \in C$, we have that 
	$$ \H^1(C,m_p\omega_C(\Sigma)^{\otimes m})= 0$$ 
	where $m_p$ is the ideal defining the point $p$. It is equivalent by duality to 
	$$ \hom_C(m_p,\omega_C^{\otimes (1-m)}(-\Sigma))=0.$$
	
	If $p$ is singular on C, let us call $\pi:\tilde{C} \rightarrow C$ the partial normalization in $p$ of $C$. We have that (see Lemma 2.1 of \cite{Cat})
	$$ \underhom(m_{p},\cO_C) \subset \pi_{*}\cO_{\tilde{C}}$$
	therefore  
	$$\hom_{\cO_C}(m_p, \omega_C(\Sigma)^{\otimes (1-m)}(-\Sigma))\subset \H^0(\tilde{C}, \pi^*\omega_C(\Sigma)^{\otimes (1-m)}(-\Sigma)).$$
	Thus, it is enough to prove that the right hand side of the previous inclusion is trivial. Notice that if $\tilde{C}$ is connected, the same argument as before gives us the vanishing of the cohomology group (because $\pi$ is birational). The same is true if $\tilde{C}$ is disconnected as long as none of the two connected components is equal to the semistable component. 
	
	Suppose now that $p$ is a separating point and $p$ belongs to the semistable component. It is easy to see that if $\omega_C(\Sigma)^{\otimes m}$ is not globally generated in $p$, then it is not globally generated on every point of the semistable component. Therefore we can reduce ourselves to prove that the line bundle is globally generated on every smooth point, which is equivalent to the vanishing  
	$$H^0(C,\omega_C(\Sigma)^{\otimes m}(-\Sigma + p)) = 0$$
	where $p$ is a smooth point of $C$. Because $m\geq 2$, it is easy to see that the only non-trivial case is when $p$ belongs to the semistable component. 
	
	Suppose $p$ is a smooth point of a semistable component $\Gamma$ of type (2) or (3) as in \Cref{rem:poss-contrac}. We have that 
	$$\omega_C(\Sigma)^{\otimes m}(-\Sigma+p)\vert_{\Gamma}\simeq \omega_{\Gamma}(D)^{\otimes m}(p)$$
	where $D$ is a lenght $2$ divisor on $\Gamma$. Therefore the restriction of the line bundle to $\Gamma$ is non-canonically isomorphic to $\cO_{\PP^1}(p)$, which is very ample. Because the curve $\Gamma$ intersects the rest of the curve in a lenght $2$ divisor, the result follows. 
	
	Moreover, suppose $p$ is a smooth point of a semistable component $\Gamma$ of type (1) as in \Cref{rem:poss-contrac}.  We have that 
	$$\omega_C(\Sigma)^{\otimes m}(-\Sigma+p)\vert_{\Gamma}\simeq \omega_{\Gamma}(p+q)^{\otimes m}(p-q)$$
	where $q$ is the point where $\Gamma$ intersects the rest of the curve. Therefore the restriction of the line bundle to $\Gamma$ is non-canonically isomorphic to $\cO_{\PP^1}$ which is globally generated. Because $\Gamma$ intersects the rest of the curve in a point, the result follows. 
	
	Finally, we prove that if $L:=\omega(\Sigma)$ then $L^{\otimes m}$ is normally generated for $m \geq 4$. This is a simplified version of the proof in \cite{Knu} as the author proved the same statement for $m\geq 3$ which requires more work. For simplicity, we write $\H^0(-)$ and $\H^1(-)$ instead of $\H^0(C,-)$ and $\H^1(C,-)$.
	
	Consider the following commutative diagram for $k\geq 1$
	$$
	\begin{tikzcd}
		\H^0(L^{\otimes 2}) \otimes \H^{0}(L^{\otimes m-2}) \otimes H^0(L^{\otimes km}) \arrow[rr] \arrow[d, "\varphi_1"] &  & \H^0(L^{\otimes m})\otimes \H^0(L^{\otimes km}) \arrow[d, "\phi"] \\
		\H^0(L^{\otimes m-2}) \otimes \H^0(L^{\otimes km+2}) \arrow[rr, "\varphi_2"]                                      &  & \H^0(L^{\otimes (k+1)m});                                        
	\end{tikzcd}
	$$
	we need to prove that $\phi$ is surjective, therefore it is enough to prove that $\varphi_1$ and $\varphi_2$ are surjective. As both $L^{\otimes 2}$ and $L^{\otimes m-2}$ are globally generated because $m\geq 4$, we can use the Generalized Lemma of Castelnuovo (see pag 170 of \cite{Knu}) and reduce to prove that $\H^1(L^{\otimes (km-2)})$ and $\H^1(L^{\otimes (km+4-m)})$ are zero for $k\geq 1$. The vanishing results follow from the fact that $km-2\geq 2$ and $km+4-m \geq 2$ for every $m\geq 4$ and $k\geq 1$. 
	
\end{proof}

\begin{remark}
	One can also prove that $\omega_C(\Sigma)^{\otimes m}$ is normally generated for $m\geq 3$. In fact, the same proof of Theorem 1.8 of \cite{Knu} can be perfectly adapted to our generality.
\end{remark}

\begin{corollary}\label{cor:contrac-field}
    Let $(C,\Sigma)$ be a $A_r$-semistable pair over an algebraically closed field. If we denote by $C'$ the schematic image of the morphism $\phi_4:C\rightarrow \PP(\H^0(\omega_C(\Sigma)^{\otimes 4}))$, then the morphism $c:C \rightarrow C'$ induced by $\phi_4$, is a contraction.
\end{corollary}

\begin{proof}
If $\omega_C(\Sigma)$ is ample, i.e. there are no semistable components, then $\omega_C(\Sigma)^{\otimes 4}$ is ample and normally generated, therefore very ample. In this case, the morphism $\phi_4$ is a closed embedding and $c$ is an isomorphism. Suppose $\Gamma$ is the semistable component. \Cref{prop:boundedness} assures us that $\phi_4$ is globally defined and it is easy to see that it contracts $\Gamma$ to a point. Because $\omega_C(\Sigma)^{\otimes 4}\vert_{\Gamma}$ is non-canonically isomorphic to $\cO_{\PP^1}$, we have that $$c\vert_{C\setminus \Gamma}:C\setminus \Gamma \longrightarrow C' \setminus c(\Gamma)$$   
is an isomorphism. We need to prove that $C'$ is a $A_r$-prestable curve and that $c^{-1}(c(\Gamma))\simeq \Gamma$. We will prove the result only in the case (3) of \Cref{rem:poss-contrac}, as the other two are classical. 

Let $p$ be the tacnode on $C$ where $\Gamma$ intersects the rest of the curve and $q$ be the image through $c$, i.e. $q:=c(p)$. We have the morphism of local rings
$$ \phi_4: \cO_{\PP(\H^0(\omega_C(\Sigma)^{\otimes 4})),q} \longrightarrow \cO_{C,p}$$
and the image of this morphism is exactly the local ring $\cO_{C',q}$. If we pass to the completions, we get the morphism 
$$ \phi_4: k[[x_1,\dots,x_m]] \longrightarrow k[[x,y]]/(y^2-x^4).$$

We identify the target of the morphism with the subalgebra of $k[[t]]\times k[[t]]$ where the two polynomials coincide up to degree $1$. Thus because  
$\omega_C(\Sigma)^{\otimes 4}\vert_{\Gamma}\simeq \cO_{\PP^1}$, we have that $\phi_4(x_i)\vert_{\Gamma}=x_i(p)$, i.e. the section $x_i$ at the point $p$. As $c$ is birational, it is clear that the image of $\phi_4$ is isomorphic to the subalgebra of $k[[t]]\times k[[t]]$ with elements of the form $(p(t),p(0))$ where $p(t) \in k[[t]]$ such that $p'(0)=0$. Therefore the completion of $\cO_{C',q}$ is a $A_2$-singularity, i.e. the complete local ring of a cusp. In fact, up to isomorphism we have that the morphism $c$ is locally of the form
$$ k[[x,y]]/(y^2-x^3) \hookrightarrow k[[x,y]]/(y^2-x^4) $$
where $(x,y) \mapsto (\frac{y+x^2}{2}, \frac{x(y+x^2)}{2})$. Finally, this implies $c^{-1}(q)\simeq \PP^1$.
 \end{proof}

\begin{remark}
 Notice that this definition of contraction is not enough to have unicity. For example, suppose we are in case (3) of \Cref{rem:poss-contrac}. \Cref{cor:contrac-field} implies that we have a contraction $$c:C \longrightarrow C'$$ of a pair $(C,\Sigma)$ with an almost-stable bridge that is contracted to a cusps. We can then find a map from $g:C' \rightarrow C''$ where $C''$ is still $A_r$-prestable, the map $g$ is an isomorphism outside the cusp of $C'$, the image of the cusp through $g$ is a $A_4$-singularity and locally on the cusps is defined as the morphism 
 $$ g: k[[x,y]]/(y^2-x^5) \longrightarrow k[[x,y]]/(y^2-x^3) $$ 
 such that $(x,y) \mapsto (x,xy)$. The composition $g\circ c$ is still a contraction. The difference is that while $C'$ has the same arithmetic genus of $C$, the same is not true for $C''$.
\end{remark}

The previous remark motivates the following definition.
\begin{definition}
 A contraction $c:C \rightarrow C'$ of a $A_r$-semistable pair $(C,\Sigma)$ is called \emph{minimal} if $g(C)=g(C')$.
\end{definition}

\begin{remark}\label{rem:expl-contrac}
 It is easy to see that a minimal contraction of a $A_r$-semistable pair sends rational tails to smooth points, rational bridges to nodes and almost bridges to cusps. In particular, there is only one minimal contraction up to isomorphism over an algebraically closed field and the almost-bridge contraction can be described \'etale-locally as
 $$ k[x,y]/(y^2-x^3) \longrightarrow k[x,y]/(y^2-x^4)$$ 
 defined by the association $(x,y) \mapsto (\frac{y+x^2}{2}, \frac{x(y+x^2)}{2})$.
 
 The following picture describes the three possible minimal contraction which can appear over an algebraically closed field. 
 \begin{figure}[H]
		\caption{Contraction morphisms}
		\centering
		\includegraphics[width=1\textwidth]{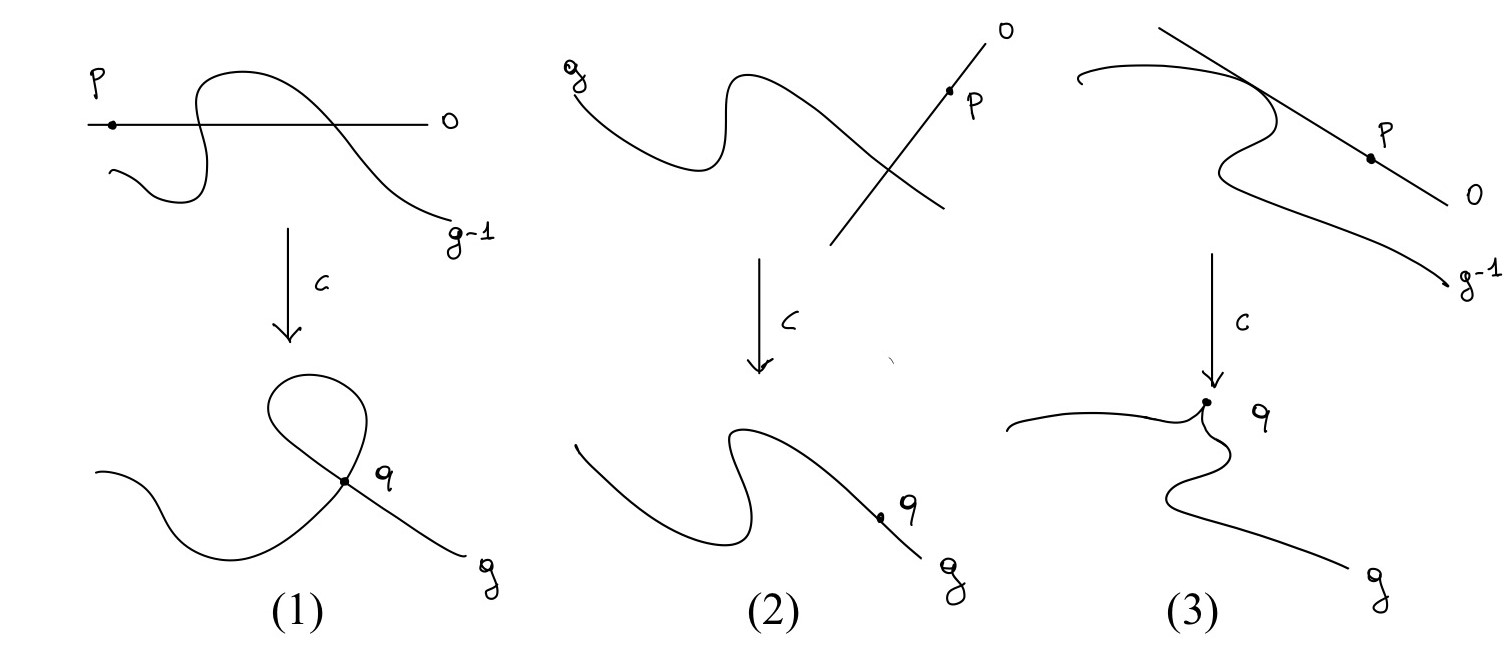}
		\label{fig:contrac}
	\end{figure}
The first one represents the minimal contraction of a rational bridge, the second one represents the minimal contraciton of a rational tail and the third one represents the contraction of an almost bridge.
\end{remark}

Before going to analyze unicity of contractions, we introduce the same notion for families. 

\begin{definition}
 A $A_r$-prestable curve $C$  of genus $g$ over $S$ is a proper, flat, finitely presented scheme over $S$ such that for every geometric point $s \in S$ the fiber $C_s$  is an $A_r$-prestable curve of genus $g$.
 
 A $A_r$-semistable pair $(C/S,\Sigma)$ of genus $g$ over a base scheme $S$ is a $A_r$-prestable curve  of genus $g$ over $S$ together with an effective  Cartier divisor $\Sigma \subset C$ over $S$ such that for every geometric point $s \in S$ the fiber $(C_s, \Sigma_s)$ is a $A_r$-semistable pair. In the same way, we can define $A_r$-stable pairs and $A_r$-stable curves.
 
 A contraction $c:C \rightarrow C'$ of a $A_r$-semistable pair $(C,\Sigma)$ over $S$ is a morphism or $A_r$-prestable curves such that for every geometric point $s \in S$ the fiber $c_s: C_s \rightarrow C_s'$ is a contraction. A contraction over $S$ is called minimal if it is minimal for every geometric fiber. 
\end{definition}

\begin{remark}
 Notice that because we require flatness for $C'$, if a contraction over $S$ is minimal over a geometric point $s \in S$ then it minimal over the whole family.
\end{remark}

\begin{remark}\label{rem:cart-div-fam}
 It is important to remark that an effective Cartier divisor $\Sigma \subset C$ over $S$ is a flat, proper, finitely presented morphism $\Sigma \rightarrow S$ such that for every geometric point $s \in S$ the fiber $\Sigma_s$ is a Cartier divisor of $C_s$. The flatness over $S$ is essential, and in fact we can use Corollary 1.5 of \cite{Knu} to prove that the image of $\Sigma$ through a contraction $c:C \rightarrow C'$ over $S$ is still flat over $S$ and satisfy base change. In particular $c(\Sigma)$ is a Cartier divisor of $C'$ over $S$ because of \Cref{rem:cart-div}.
 
 Moreover, because $\Sigma$ restricted to every non-trivial fiber of $c$ is at most of degree $1$, this implies that $R^1c_*\cO(-\Sigma)=0$ and in particular $Rc_*\cO(-\Sigma)\simeq\cO(-c(\Sigma))$ as complexes of quasi-coherent modules over $C'$.
\end{remark}

We are finally ready to state the unicity result for the minimal contractions. Firstly, we have to prove the Main Lemma from Knutsen's work in our generality (see Lemma 1.6 of \cite{Knu}). 

\begin{lemma}\label{lem:main-lemma}
 Let $c:C \rightarrow C'$ be a minimal contraction of a $A_r$-semistable pair $(C,\Sigma)$ over $S$ and denote by $\Sigma'$ the image $c(\Sigma)\subset C'$. Then for every $k\geq 1$
 \begin{itemize}
  \item[(a)] there are canonical isomorphisms (commute with base change over $S$)
  $$ c^*(\omega_{C'}(\Sigma')^{\otimes k}) \simeq \omega_{C}(\Sigma)^{\otimes k}$$
  and 
   $$ \omega_{C'}(\Sigma')^{\otimes k} \simeq c_*(\omega_{C}(\Sigma)^{\otimes k});$$
   \item[(b)] $ R^1c_*(\omega_{C}(\Sigma)^{\otimes k})=0$;
   \item[(c)] if we denote by $\pi:C\rightarrow S$ and $\pi':C' \rightarrow S$ the two structural morphisms, then $R^i\pi_*\omega_{C}(\Sigma)^{\otimes k}\simeq R^i\pi'_*\omega_{C'}(\Sigma')^{\otimes k}$ for $i=0,1$. 
 \end{itemize}

\end{lemma}

\begin{proof}
 Because $\omega_{C}(\Sigma)^{\otimes k}$ restricted to the non-trivial fibers of $c$ is non-canonically isomorphic to $\cO_{\PP^1}$, both (b) and (c) follows from (a) and Corollary 1.5 of \cite{Knu}. Notice also that if we have the canonical isomorphism 
 $$ c^*(\omega_{C'}(\Sigma')^{\otimes k}) \simeq \omega_{C}(\Sigma)^{\otimes k}$$
 we can then pushforward it through $c$ and precomposing it with the counit of the adjunction $c^* \dashv c_*$. However, if $\cF$ is a locally free sheaf on $C'$, we have that 
 $$ \cF \longrightarrow c_*c^*\cF\simeq \cF \otimes c_*\cO_C$$ 
 thanks to projection formula. We get that the counit of the adjunction is an isomorphism for locally free sheaves because the geometric fiber of $c$ are connected and reduced. 
 
 Finally, we have reduced to construct the canonical isomorphism 
 $$c^*(\omega_{C'}(\Sigma')^{\otimes k}) \simeq \omega_{C}(\Sigma)^{\otimes k},$$
 and, because tensor and pullback commute, it is enough to construct it for $k=1$. 
 
 We start by constructing a canonical morphism and then we prove it is an isomorphism. We claim that 
 $$ \pi_*\curshom_C(c^*(\omega_{C'}(\Sigma')),\omega_{C}(\Sigma)) \simeq \cO_S;$$ 
 if it is true, we can choose $1 \in \H^0(S,\cO_S)$ to be our canonical morphism which will be denoted by $\varphi_c:c^*(\omega_{C'}(\Sigma'))\rightarrow \omega_{C}(\Sigma)$. 
 
 First of all, Grothendieck duality gives us an isomorphism  
 $$ R\pi_*R\curshom_C(Lc^*(\omega_{C'}(\Sigma'))(-\Sigma),\omega_{C}[1]) \simeq R\curshom_S(R\pi_*(Lc^*(\omega_{C'}(\Sigma'))(-\Sigma)),\cO_S).$$
 
 Let us study the two complexes. Notice that $Lc^*(\omega_{C'}(\Sigma'))$ is isomorphic to $\omega_{C'}(\Sigma')$ as a complex concentrated in degree $0$. Therefore the left hand side is a two-term complex concentrated in $[-1,0]$ such that 
  $$\H^{-1}\Big(R\pi_*R\curshom_C(Lc^*(\omega_{C'}(\Sigma'))(-\Sigma),\omega_{C}[1]\Big)\simeq \pi_*\curshom_C(c^*(\omega_{C'}(\Sigma')),\omega_{C}(\Sigma))$$ 
 while 
  $$ \H^0\Big(R\pi_*R\curshom_C(Lc^*(\omega_{C'}(\Sigma'))(-\Sigma),\omega_{C}[1]\Big) \simeq R^1\pi_*\curshom_C(c^*(\omega_{C'}(\Sigma')),\omega_{C}(\Sigma)).$$
Regarding the right hand side, first of all we have that 
  $$Rc_*(c^*(\omega_{C'}(\Sigma'))(-\Sigma)) \simeq \omega_{C'}(\Sigma') \otimes^{\LL} Rc_*\cO(-\Sigma) \simeq \omega_{C'}$$
  where the first isomorphism is the (derived) projection formula and the second isomorphism follows from \Cref{rem:cart-div-fam}. Because $R\pi_*\simeq R\pi'_* \circ Rc_*$, one easily get that the left hand side is also a complex concentrated in $[-1,0]$ as expected and
  $$ \H^{-1}\Big(R\curshom_S(R\pi_*(Lc^*(\omega_{C'}(\Sigma'))(-\Sigma)),\cO_S)\Big) \simeq \cO_S$$
  while 
  $$ \H^{0}\Big(R\curshom_S(R\pi_*(Lc^*(\omega_{C'}(\Sigma'))(-\Sigma)),\cO_S)\Big) \simeq \pi'_*\omega_{C'}.$$
  This proves the claim. 
  
  By construction, if $c$ is an isomorphism in a point $p \in C$ then $\varphi_c$ is an isomorphism in $p$ as well. Clearly, it is enough to prove that $\varphi_c$ is surjective and in particular we can reduce to prove it restricting the morphism to the geometric fibers of $C\rightarrow S$. This follows from \Cref{rem:expl-contrac} and an \'etale-local computation.
\end{proof}

\begin{proposition}\label{prop:unique-contrac}
 Let $(C,\Sigma)$ be a $A_r$-semistable pair over $S$. Then there exists a unique minimal contraction up to unique isomorphism, i.e. if $c_1:C \rightarrow C_1$ and $c_2:C \rightarrow C_2$ are two minimal contractions, then there exists a unique $f:C_1\rightarrow C_2$ which makes the following diagram
 $$
\begin{tikzcd}
                    & C \arrow[ld, "c_1"'] \arrow[rd, "c_2"] &     \\
C_1 \arrow[rr, "f"] &                                        & C_2
\end{tikzcd}
$$ 
commute.
\end{proposition}

\begin{proof}
 The unicity of $f$ follows from the fact that a contraction is an epimorphism. Let us focus on the existence part. In particular, we will prove that any minimal contraction is isomorphic to the one defined in \Cref{cor:contrac-field}. 
 
 Consider a minimal contraction $c:C\rightarrow C'$ and recall that $(C',\omega_{C'}(\Sigma'))$ is a $A_r$-stable pair, where $\Sigma'=c(\Sigma)$ (see \Cref{rem:cart-div}). \Cref{prop:boundedness} implies that 
 $$ C' \simeq \proj_S \Big(\bigoplus_{n \in \NN} \pi'_*(\omega_{C'}(\Sigma')^{\otimes 4n})\Big) \hookrightarrow \PP(\pi'_*\omega_{C'}(\Sigma')^{\otimes 4})$$ 
 where $\pi':C' \rightarrow S$ is the structural morphism. Moreover, \Cref{lem:main-lemma} gives us the following isomorphism 
 $$ \pi_*\omega_{C}(\Sigma)^{\otimes k}\simeq \pi'_*\omega_{C'}(\Sigma')^{\otimes k}$$
 which clearly implies that $C'$ coincides with the contraction defined in \Cref{cor:contrac-field}, i.e. the image of the morphism 
 $$ C\longrightarrow \PP(\pi_*\omega_C(\Sigma)^{\otimes 4}).$$
 The image of such morphism is a $A_r$-prestable curve over $S$ because it is flat over $S$ thanks to \Cref{prop:boundedness} and the geometric fibers have been described in \Cref{cor:contrac-field}.
\end{proof}

We end up this section with a technical lemma. It helps us dealing with the case when we have two $A_r$-semistable pairs that contracts to the same curve.

\begin{lemma}\label{lem:contrac-unique}
 Let $(C_1,\Sigma_1)$ and $(C_2,\Sigma_2)$ be two $A_r$-semistable pairs over an algebraically closed field such that $(C_1,\Sigma_1)$ is not $A_r$-stable. Suppose we have the following diagram
 $$
 \begin{tikzcd}
C_1 \arrow[rd, "c_1"] \arrow[rr, "g"] &    & C_2 \arrow[ld, "c_2"'] \\
& C' &                       
\end{tikzcd}
$$
 such that 
 \begin{itemize}
  \item $c_1$, $c_2$ are minimal contractions;
  \item if $\Gamma_1$ is the semistable component of $C_1$, then $g\vert_{\Gamma_1}:\Gamma_1 \rightarrow C_2$ is a closed embedding.
 \end{itemize}
Then $g$ is an isomorphism. 
\end{lemma}

\begin{proof}
    First of all, notice that by \Cref{rem:poss-contrac} we have that $C_1$ and $C_2$ have the same type of semistable components ($\Gamma_1$ and $\Gamma_2$ respectively) and clearly $g(\Gamma_1)\subset\Gamma_2$. Because $g\vert_{\Gamma_1}$ is a closed embedding, then it is clear that $g\vert_{\Gamma_1}: \Gamma_1 \rightarrow \Gamma_2$ is an isomorphism. In the same way, we have that $g$ restricted to $C_1\setminus \Gamma_1$ is an isomorphism with $C_2\setminus \Gamma_2$.  The statement follows easily from a local computation.
\end{proof}

\section{Moduli stack of $A_r$-stable curves}\label{sec:2}
Let $g$ be an integer with $g\geq 2$, $r$ be a nonnegative integer and $\kappa$ be a base field with characteristic greater than $r+1$.

We denote by $\Mtilde_{g}^r$ the category defined in the following way: an object is a proper flat finitely presented morphism $C\arr S$ over $\kappa$ such that every geometric fiber over $S$ is a $A_r$-stable curve of genus $g$. These families are called \emph{$A_r$-stable curves} over $S$. Morphisms are defined in the usual way. This is clearly a fibered category over the category of schemes over $\kappa$. 

Fix a positive integer $n$. In the same way, we can define $\Mtilde_{g,n}^r$ the fibered category whose objects are the datum of $A_r$-stable curves over $S$ with $n$ distinct sections $p_1,\dots,p_n$ such that every geometric fiber over $S$ is a  $n$-pointed $A_r$-stable curve. These families are called \emph{$n$-pointed $A_r$-stable curves} over $S$. Morphisms are just morphisms of $n$-pointed curves.

The main resultx of this section are the description of $\Mtilde_{g,n}^r$ as a quotient stack and the description of the contraction morphism. Firstly, we need to prove two results which are classical in the case of ($A_1$-)stable curves.

\begin{proposition}\label{prop:openness}
	Let $C \arr S$ a proper flat finitely presented morphism with $n$-sections $s_i:S \arr C$ for $i=1,\dots,n$. There exists an open subscheme $S' \subseteq S$ with the property that a morphism $T \arr S$ factors through $S'$ if and only if the projection $T\times_{S} C \arr T$, with the sections induced by the $s_{i}$, is a $n$-pointed $A_r$-stable curve.
\end{proposition}

\begin{proof}
	It is well known that a small deformation of a curve with $A_h$-singularities for $h\leq r$ still has $A_h$-singularities for $h\leq r$. Hence, after restricting to an open subscheme of $S$ we can assume that $C \arr S$ is an $A_{r}$-prestable curve. By further restricting $S$ we can assume that the sections land in the smooth locus of $C \arr S$, and are disjoint. Then the result follows from openness of ampleness for invertible sheaves. 
\end{proof}

We are ready to prove the theorem. To be precise, both the proofs of the following theorem and of \Cref{prop:openness} are an adaptation of Theorem 1.3 and of Proposition 1.2 of \cite{DiLorPerVis} to the more general case of $A_r$-stable curves. 

\begin{theorem}\label{theo:descr-quot}
	$\Mtilde_{g,n}^r$ is a smooth connected algebraic stack of finite type over $\kappa$. Furthermore, it is a quotient stack: that is, there exists a smooth quasi-projective scheme X with an action of $\GL_N$ for some positive $N$, such that 
	$ \Mtilde_{g,n}^r \simeq [X/\GL_N]$. 
\end{theorem}

\begin{proof}
	It follows from \Cref{prop:boundedness} that if $(\pi\colon C \arr S,\Sigma_1,\dots,\Sigma_n)$ is a $n$-pointed $A_{r}$-stable curve of genus $g$, then $\pi_{*}\omega_{C/S}(\Sigma_{1}+ \dots + \Sigma_{n})^{\otimes 3}$ is locally free sheaf of rank $N:=5g-5+3n$, and its formation commutes with base change, because of Grothendieck's base change theorem. 
	
	Call $X$ the stack over $\kappa$, whose sections over a scheme $S$ consist of a $n$-pointed  $A_{r}$-stable curve as above, and an isomorphism $\cO_{S}^{N} \simeq \pi_{*}\omega_{C/S}(\Sigma_{1}+ \dots + \Sigma_{n})^{\otimes m}$ of sheaves of $\cO_{S}$-modules. Since $\pi_{*}\omega_{C/S}(\Sigma_{1}+ \dots + \Sigma_{n})^{\otimes m}$ is very ample , the automorphism group of an object of $X$ is trivial, and $X$ is equivalent to its functor of isomorphism classes. We proved the line bundle is very ample for $m\geq 4$, but it is still true for $m=3$. Clearly, the same prove also gives us the result for $n=4$, except that $N$ increases.
	
	Call $H$ the Hilbert scheme of subschemes of $\PP^{N-1}_{\kappa}$ with Hilbert polynomial $P(t)$, and $D \arr H$ the universal family. Call $F$ the fiber product of $n$ copies of $D$ over $S$, and $C \arr F$ the pullback of $D \arr H$ to $F$; there are $n$ tautological sections $s_{1}$, \dots,~$s_{n}\colon F \arr C$. Consider the largest open subscheme $F'$ of $F$ such that the restriction $C'$ of $C$, with the restrictions of the $n$ tautological sections, is a $n$-pointed $A_{r}$-stable curve, as in Proposition~\ref{prop:openness}. Call $Y \subseteq F'$ the open subscheme whose points are those for which the corresponding curve is nondegenerate, $E \arr Y$ the restriction of the universal family, $\Sigma_{1}$, \dots,~$\Sigma_{n} \subseteq E$ the tautological sections. Call $\cO_{E}(1)$ the restriction of $\cO_{\PP^{N-1}_{Y}}(1)$ via the tautological embedding $E \subseteq \PP^{N-1}_{Y}$; there are two section of the projection $\pic_{E/Y}^{3(2g-2 + n)}\arr Y$ from the Picard scheme parametrizing invertible sheaves of degree $3(2g-2 + n)$, one defined by $\cO_{E}(1)$, the other by $\omega_{E/Y}(\Sigma_{1} + \dots + \Sigma_{n})^{\otimes 3}$; let $Z \subseteq Y$ the equalizer of these two sections, which is a locally closed subscheme of $Y$.
	
	Then $Z$ is a quasi-projective scheme over $\kappa$ representing the functor sending a scheme $S$ into the isomorphism class of tuples consisting of a $n$-pointed  $A_{r}$-stable curve $\pi\colon C \arr S$, together with an isomorphism of $S$-schemes
	\[
	\PP^{N-1}_{S} \simeq \PP(\pi_{*}\omega_{C/S}(\Sigma_{1} + \dots + \Sigma_{n})^{\otimes 3})\,.
	\]
	There is an obvious functor $X \arr Z$, associating with an isomorphism $\cO_{S}^{N} \simeq \pi_{*}\omega_{C/S}(\Sigma_{1}+ \dots + \Sigma_{n})^{\otimes 3}$ its projectivization. It is immediate to check that $X \arr Z$ is a $\gm$-torsor; hence it is representable and affine, and $X$ is a quasi-projective scheme over $\spec \kappa$.
	
	On the other hand, there is an obvious morphism $X \arr \mt_{g,n}^r$ which forgets the isomorphism $\cO_{S}^{N} \simeq \pi_{*}\omega_{C/S}(\Sigma_{1}+ \dots + \Sigma_{n})^{\otimes 3}$; this is immediately seen to be a $\GL_{N}$ torsor. We deduce that $\mt_{g,n}^r$ is isomorphic to $[X/\GL_{N}]$. This shows that is a quotient stack, as in the last statement; this implies that $\mt_{g,n}^r$ is an algebraic stack of finite type over $\kappa$.
	
	The fact that $\mt_{g,n}^r$ is smooth follows from the fact that deformations of $A_{r}$-prestable curves are unobstructed.
	
	Finally, to check that $\mt_{g,n}^r$ is connected it is enough to check that the open embedding $\cM_{g,n} \subseteq \mt_{g,n}^r$ has a dense image, since $\cM_{g,n}$ is well known to be connected. This is equivalent to saying that every $n$-pointed $A_r$-stable curve over an algebraically closed extension $\Omega$ of $\kappa$ has a small deformation that is smooth. Let $(C, p_{1}, \dots, p_{n})$ be a $n$-pointed $A_r$-stable curve; the singularities of $C$ are unobstructed, so we can choose a lifting $\overline{C}\arr \spec \Omega\ds{t}$, with smooth generic fiber. The points $p_{i}$ lift to sections $\spec\Omega\ds{t} \arr \overline{C}$, and then the result follows from  Proposition~\ref{prop:openness}.
	
\end{proof}
\begin{remark}\label{rem: max-sing}
	Clearly, we have an open embedding $\Mtilde_{g,n}^r  \subset \Mtilde_{g,n}^s$ for every $r\leq s$. Notice that $\Mtilde_{g,n}^r=\Mtilde_{g,n}^{2g+1}$ for every $r\geq 2g+1$. 
\end{remark}

We prove that the usual definition of Hodge bundle extends to our setting.  As a consequence we obtain a locally free sheaf $\HH_{g}$ of rank~$g$ on $\mt_{g, n}^r$, which is called \emph{Hodge bundle}.

\begin{proposition}\label{prop:hodge-bundle}
	Let $\pi\colon C \arr S$ be a $n$-pointed $A_r$-stable of genus $g$. Then $\pi_{*}{\omega}_{C/S}$ is a locally free sheaf of rank $g$ on $S$, and its formation commutes with base change.
\end{proposition}
\begin{proof}
	If $C$ is an $A_r$-stable curve of genus $g$ over a field $k$, the dimension of $\H^{0}(C, \omega_{C/k})$ is $g$; so the result follows from Grauert's theorem when $S$ is reduced. But the versal deformation space of an $A_{r}$-stable curve over a field is smooth, so every $A_{r}$-stable curve comes, \'{e}tale-locally, from an $A_{r}$-stable curve over a reduced scheme, and this proves the result.
\end{proof}

Finally, we describe the contraction morphism. Luckily, most of the work is done in \Cref{sec:1}.

\begin{theorem}\label{theo:contrac}
	We have a morphism of algebraic stacks 
	$$ \gamma:\Mtilde_{g,n+1}^r \longrightarrow \Ctilde_{g,n}^r$$
	where $\Ctilde_{g,n}^r$ is the universal curve of $\Mtilde_{g,n}^r$. Furthermore, it is an open immersion and its image is the open locus $\Ctilde_{g,n}^{r,\leq 2}$ in $\Ctilde_{g,n}^r$ parametrizing $n$-pointed $A_r$-stable curves $(C,p_1,\dots,p_n)$ of genus $g$ and a (non-necessarily smooth) section $q$ such that $q$ is an $A_h$-singularity for $h\leq 2$.
\end{theorem}

\begin{proof}
	Let $(C,p_1,\dots,p_{n+1})$ be a $n+1$-pointed $A_r$-stable curve over $S$. If we denote by $\Sigma$ the Cartier divisor $\cO(p_1+\dots+p_n)$ on $C$, it is clear that $(C,\Sigma)$ is a $A_r$-semistable pair over $S$. Therefore we can construct the (minimal) contraction functor $\gamma$ thanks to \Cref{prop:unique-contrac} and it factors through the open $\Ctilde_{g,n}^{r,\leq 2}$ thanks to \Cref{cor:contrac-field}.
	
	We can construct an inverse to this map, namely a functor
	$$ \eta:\Ctilde_{g,n}^{r,\leq 2} \longrightarrow \Mtilde_{g,n+1}^r$$ 
	which is a quasi-inverse of $\gamma$. For the sake of notation, we denote by $\cC$ the stack $\Ctilde_{g,n}^{r,\leq 2}$. Consider the universal family of $\cD \rightarrow \cC$ of $\cC$ with universal sections $(q_1,\dots,q_n,q_{n+1})$. Furthermore, consider the closed substack $\Delta$ of $\cC$ parametrizing $(\tilde{C},q_1,\dots,q_n,q_{n+1})$ such that $q_{n+1}$ lands in the singular locus of the morphism $\tilde{C}\rightarrow S$. We can define its stacky structure using the $1$-st Fitting ideal of $\Omega_{\tilde{C}/S}$. 
	
	We denote by $E_i$ the closed substack $q_{n+1} \cap q_{i}$ of $\cD$ for $i=1,\dots,n$ and by $E_0$ the closed substack $q_{n+1}(\Delta)$. We have that the $E_i$'s are disjoint closed substacks of $\cD$ of codimension $2$ and we denote by $E$ the union. Finally, we denote by $\cB$ the blowup of $\cD$ with center $E$ and by $p_{i}$ the proper transform of $q_i$ for every $i=1,\dots,n+1$. We need to prove that $(\cB,p_1,\dots,p_{n+1})$ is a $n+1$-pointed $A_r$-stable curve over $\cC$. In fact, we are going to prove that the blowup morphism $\cB \rightarrow \cD$ is a (minimal) contraction for the semistable pair $(\cB,\Sigma)$ where $\Sigma\simeq \cO(p_1+\dots+p_n)$.
	
	We have that $\cB$ is a proper, finitely presented morphism over $\cC$. To prove that $\cB$ is flat over $\cC$, we notice that we have a factorization of the inclusion of the center of the blowup inside $\cD$
	$$
	\begin{tikzcd}
    E \arrow[r, hook] & \cC \arrow[r, "q_{n+1}", hook] & \cD
    \end{tikzcd}
    $$
	where $E$ is a Cartier divisor inside $\cC$. If we denote by $I_{E}$ the ideal defining $E$ in $\cD$, that we have that $I_{E}^{\otimes m}$ is flat over $\cC$ and therefore $\cB:=\proj_{\cD}\Big(\bigoplus_{m} I_{E}^{\otimes m} \Big)$ is also flat over $\cC$. 
	
	We need to prove that the geometric fibers of $\cB \rightarrow \cC$ are $A_r$-prestable curves. The stability condition then follows easily. For simplicity, we will prove the case $n=0$ but the same proof works for any $n$. Moreover, we will explicitly do the computation in a neighborhood of a geometric object in $\cC$ of the form $(\tilde{C},q)$ where $q$ is a cusp because it is the only non-classical case. The same idea can be applied to the other cases.
	
	As usual, deformation theory tells us that we can consider a smooth neighborhood $(\spec A,m_q)$ of $(\widetilde{C},q)$ in $\cC$ with $A$ a smooth algebra. This implies that we have an $A_r$-stable curve
	$$
	\begin{tikzcd}
		\widetilde{C}_A \arrow[r] & \spec A \arrow[l, "q_A"', bend right]
	\end{tikzcd}
	$$
	with $q_A$ a global section such that $q_A \otimes_A A/m_q = q$. By deformation theory of cusps (see Example 6.2.12 of \cite{TalVis}), we know that the completion of the local ring of $q$ in $\widetilde{C}_A$ is of the form
	$$ A[[x,y]]/(y^2-x^3-r_1x-r_2)$$
	where $r_1,r_2$ are part of a system of parameters for $A$ because $(\spec A, m_q)$ is versal. Consider know the element $f:=4r_1^3+27r_2^2 \in A$, which parametrizes the locus when the section $q_A$ ends in a singular point, which is either a node or a cusp. Let $q_1:\spec A/f \into \widetilde{C}_A$ be the codimension two closed immersion and denote by $C_A:={\rm Bl}_{q_1}\widetilde{C}$. We want to prove that the geometric fiber of $C_A$ over $m_q$ is an almost-bridge.
	
	Notice that the formation of the blowup does not commute with arbitrary base change. However, consider the system of parameters $(r_1,r_2,\dots,r_a)$ with $a:=\dim A=3g-2$. Then if we denote by $J:=(r_2,\dots,r_a)$ we have that 
	$$ \tor_1^A(A/f^n,A/J)=0$$
	for every $n$. This implies that the blowup commutes with the base change for $A/J$ and therefore we can reduce to the case when $A$ is a DVR, $r:=r_1$ is the uniformizer and the completion of the local ring of $q$ in $\widetilde{C}_A$ is of the form 
	$$ A[[x,y]]/(y^2-x^3-rx).$$ 
	A standard blowup computation proves that over the special fiber of $A$ we have
	\begin{itemize}
		\item $C_A\otimes_A A/m_q$ is an $A_r$-prestable curve and if we denote by $p_A$ the proper transform of $q_A$, we have that $(C_A \otimes_A A/m_q, p_A \otimes_A A/m_q)$ is an $1$-pointed $A_r$-stable curve of genus $g$,  
		\item $C_A \rightarrow \widetilde{C}_A \otimes_A A/m_q$ is the contraction of an almost-bridge.
    \end{itemize}  
	
	We have finally proved the existence of the two functors $\gamma$ and $\eta$. \Cref{prop:unique-contrac} assures us that there exists a unique isomorphism of functors $\gamma \circ \eta \simeq \id$. Regarding the other composition, the universal property of the blowup gives us that it exists a natural transformation $\id \Rightarrow \eta \circ \gamma $ and it is enough to prove that it is an isomorphism for every geometric point of $\Mtilde_{g,n+1}^r$. This follows from \Cref{lem:contrac-unique}
\end{proof}

\section{Preliminary results}\label{sec:3}
In this section, we prove several preliminary results that are used for the description of the locus parametrizing curves with $A_h$-singularities.

Firstly, we introduce the moduli stack parametrizing finite flat extensions of pointed finite flat curvilinear algebras and describe it as an affine bundle over a classifying stack.

Secondly, we discuss the behaviour of pushouts and blowups in families and we prove that the two constructions are one the inverse of the other under some strict conditions, which are satisfied in our situation.

Most of them are probably well-known to experts, but for lack of a better reference in the literature, we decided to write down the proofs.

\subsection{Extensions of finite curvilinear algebras}\label{sub:3.1}

We describe the moduli stack of finite flat extensions of curvilinear algebras.  

Let $m$ be a positive integer and $\kappa$ be the base field. We denote by $F\cH_m$ the finite Hilbert stack of a point, i.e. the stack whose objects are pairs $(A,S)$ where $S$ is a scheme and $A$ is a locally free finite $\cO_S$-algebra of degree $m$. We know that $F\cH_m$ is an algebraic stack, which is in fact a quotient stack of an affine scheme of finite type by a smooth algebraic group. For a more detailed treatment see Section 96.13 of \cite{StProj}. 

Given another positive integer $d$, we can consider a generalization of the stack $F\cH_m$. Given a morphism $\cX \rightarrow \cY$, one can consider the finite Hilbert stack $F\cH_d(\cX/\cY)$ which parametrizes commutative diagrams of the form
$$
\begin{tikzcd}
	Z \arrow[r] \arrow[d, "f"] & \cX \arrow[d] \\
	S \arrow[r]                & \cY          
\end{tikzcd}
$$
where $f$ is finite locally free of degree d. This again can be prove to be algebraic. If $\cY=\spec \kappa$, we denote it simply by $F\cH_d(\cX)$. For a more detailed treatment see Section 96.12 of \cite{StProj}.

Finally, we define $E\cH_{m,d}$ to be the fibered category in groupoids whose objects over a scheme $S$ are finite locally free extensions of $\cO_S$-algebras $A \into B$  of degree $d$ such that $A$ is a finite locally free $\cO_S$-algebra of degree $m$. Morphisms are defined in the obvious way. Clearly the algebra $B$ is finite locally free of degree $dm$. 

\begin{proposition}
	The stack $F\cH_{d}(F\cH_m)$ is naturally isomorphic to $E\cH_{m,d}$, therefore $E\cH_{m,d}$ is an algebraic stack.
\end{proposition}

\begin{proof}
	The proof follows from unwinding the definitions.
\end{proof}

We want to add the datum of a section of the structural morphism $\cO_S \into B$. This can be done passing to the universal object of $F\cH_{dm}$.

Let $n$ be a positive integer and let $B_{\rm univ}$ the universal object of $F\cH_{n}$; consider $\cF_{n}:=\spec_{F\cH_{n}}(B_{\rm univ})$ the generalized spectrum of the universal algebra over $F\cH_{n}$. It parametrizes pairs $(B,q)$ over $S$ where $B \in F\cH_{n}(S)$ and $q:B \rightarrow \cO_S$ is a section of the structural morphism $\cO_S \into B$. 
\begin{definition}
	We say that a pointed algebra $(A,p) \in \cF_n$ over an algebraically closed field $k$ is linear if $\dim_k m_p/m_p^2 \leq 1$, where $m_p$ is the maximal ideal associated to the section $p$. We say that $(A,p)$ is curvilinear if $(A,p)$ is linear and $\spec A=\{p\}$.
\end{definition}

We want to study the moduli stack parametrizing curvilinear pointed algebras of lenght $n$. This is locally closed inside $\cF_n$ and we endow it with a stacky structure in the following way.

We consider the closed substack defined by the $1$-st Fitting ideal of $\Omega_{A|\cO_S}$ in $\cF_{n}$. This locus parametrizes non-linear pointed algebras . Therefore we can just consider the open complement, which we denote by $\cF_{n}^{\rm lin}$. We can inductively define closed substacks of $\cF_{n}^{\rm lin}$ in the following way: suppose $\cS_h$ is defined, then we can consider $\cS_{h+1}$ to be the closed substack of $\cS_h$ defined by the $0$-th fitting ideal of $\Omega_{\cS_h|F\cH_{n}}$. We set $\cS_1=\cF_{n}^{\rm lin}$. It is easy to prove that the geometric points of $\cS_h$ are pairs $(A,p)$ such that $A$ localized at $p$ has length greater or equal than $h$. Notice that this construction stabilizes at $h=n$ and the geometric points of $\cS_h$ are exactly the curvilinear pointed algebras. Finally, we denote by $\cF_{n}^{c}:=\cS_{n}$ the last stratum. As it is a locally closed substack of an algebraic stacks of finite type, it is algebraic of finite type too. 

\begin{lemma}\label{lem:loc-triv-alg}
	If $(B,q) \in \cF_{n}^c(S)$ for some $\kappa$-scheme $S$, then it exists an \'etale cover $S'\rightarrow S$ such that $B\otimes_S S' \simeq \cO_{S'}[t]/(t^{n})$ and $q\otimes_S S'=q_0\otimes S'$ where
	$$q_0:\kappa[t]/(t^{n}) \longrightarrow \kappa$$
	is defined by the assocation $t\mapsto 0$.
\end{lemma}

\begin{proof}
	We are going to prove that $\cF_{n}^c$ has only one geometric point (up to isomorphism) and its tangent space is trivial. Thus the thesis follows, by a standard argument in deformation theory.  
	
	Suppose then $S=\spec k$ is the spectrum of an algebraically closed field and $B$ is a finite $k$-algebra (of degree $n$). Because $(B,q)$ is linear, we have that $\dim_k(m_q/m_q^2)\leq 1$. We can construct then a surjective morphism 
	$$ k[[t]] \longrightarrow B_{m_q}$$
	whose kernel is generated by the monomial $t^{n'}$. We have then $(B,q)\simeq (k[t]/(t^{n'}),q_0)$ for some positive integer $n'$. Because $B$ is local of length $n$, we get that $n'=n$.

	Let us study the tangent space of $\cF_n^c$ at the pointed algebra $(k[t]/(t^{n}),q_0)$ where $k/\kappa$ is a field. We have that any deformation $(B_{\epsilon},q_{\epsilon})$ of the pair $(k[t]/(t^n),q_0)$ is of the form
	$$ k[t,\epsilon]/(p(t,\epsilon))$$
	where $p(t,0)=t^{n}$. Because the section is defined by the association $t\mapsto b\epsilon$, we have also that $p(b\epsilon,\epsilon)=p(0,\epsilon)=0$. It is easy to see that $(B_{\epsilon},q_{\epsilon}) \in \cF^c_{n}$ only if $$p(b\epsilon,\epsilon)=p'(b\epsilon,\epsilon)=\dots=p^{(n-1)}(b\epsilon,\epsilon)=0$$
	where the derivatives are done over $k[\epsilon]/(\epsilon^2)$, thus $p(t,\epsilon)=(t-b\epsilon)^{n}$. The algebra obtained is clearly isomorphic to trivial one.
\end{proof} 

\begin{remark}
 Notice that we do not need the field $k$ to be algebraically closed to prove that there is only one pointed curvilinear algebra of lenght $n$ over $k$. This is related to the fact that the automorphism group of the unique point is special and therefore the previous statement is true even Zariski-locally. 
\end{remark}

Let $G_{n}:=\underaut\Big(\kappa[t]/(t^{n}),q_0\Big)$ be the automorphism group of the standard pointed curvilinear algebra. One can describe $G$ as the semi-direct product of $\gm$ and a group $U$ which is isomorphic to an affine space of dimension $n-2$.

\begin{corollary}\label{cor:descr-finite-alg}
	We have that $\cF_{n}^c$ is isomorphic to $\cB G_{n}$, the classifying stack of the group $G_{n}$.
\end{corollary}

We denote by $\cE_{m,d}^c$ the fiber product $E\cH_{m,d}\times_{F\cH_{dm}}\cF_{dm}^c$. We get the morphism of algebraic stacks
$$ \cE_{m,d}^c \longrightarrow \cF_{dm}^c$$ 
defined by the association $(S,A\into B \rightarrow \cO_S) \mapsto (S,B\rightarrow \cO_S)$. Notice that the morphism is faithful, therefore representable by algebraic spaces.  Clearly the composite $A\rightarrow \cO_S$ is still a section. 

We now study the stack $\cE_{m,d}^c$.

\begin{lemma}\label{lem:triv-ext}
	If $(A\into B,q) \in \cE_{m,d}^c(S)$ for some scheme $k$-scheme $S$, then it exists an \'etale cover $\pi:S'\rightarrow S$ such that 
	$$\pi^*\Big(A\into B,q\Big) \simeq \Big(\phi_d:\cO_{S'}[t]/(t^{m})\into \cO_{S'}[t]/(t^{dm}),q_0\otimes_{S}S'\Big)$$
	where $\phi_d(t)=t^dp(t)$ with $p(0) \in \cO_{S'}^{\times}$.
\end{lemma}

\begin{proof}
	First of all, an easy computation shows that any finite flat extension of pointed algebras of degree $d$ $$\cO_{S}[t]/(t^{m})\into \cO_{S}[t]/(t^{dm})$$ 
	is of the form $\phi_d$ for any scheme $S$. Therefore, it is enough to prove that if $A\into B$ is finite flat of degree $d$, then $B$ curvilinear implies $A$ curvilinear. In analogy with the proof of \Cref{lem:loc-triv-alg}, we prove the statement for $S=\spec k$ and then for $S=\spec k[\epsilon]/(\epsilon^2)$ with $k$ algebraically closed field.
	
	Firstly, suppose $S=\spec k$ with $k$ algebraically closed. By \Cref{lem:loc-triv-alg}, we know that $(B,q)\simeq (k[t]/(t^{dm}),q_0)$. We need to prove now that also $A$ is curvilinear. Clearly $A$ is local because of the going up property of flatness and we denote by $m_A$ its maximal ideal. 
	
	If we tensor $A\into B$ by $A/m_A$, we get $k\into k[t]/m_Ak[t]$. However flatness implies that the extension $k\into k[t]/(m_Ak[t])$ has degree $d$, therefore because $k[t]$ is a PID, it is clear that $m_Ak[t]\subset m_q^{d}$ and the morphism $m_A/m_A^2\rightarrow m_q^d/m_q^{d+1}$ is surjective. Let $a \in m_A$ be an element whose image is $t^d$ modulo $t^{d+1}$. If we now consider $A/a\into B/aB\simeq k[t]/(t^d)$, by flatness it still has length $d$, but this implies $A/(a)\simeq k$ or equivalently $m_A=(a)$. Therefore $A$ is curvilinear too. 
	
	Suppose now $S=\spec k[\epsilon]$.
	We know that, given a morphism of schemes $X\rightarrow Y$, we have the exact sequence
	$$ {\rm Def}_{X\rightarrow Y} \longrightarrow {\rm Def}_X \oplus {\rm Def}_Y \longrightarrow \ext^1_{\cO_X}(Lf^*{\rm NL}_Y,\cO_X)$$
	where ${\rm NL}_Y$ is the naive cotangent complex of $Y$. We want to prove that if $X\rightarrow Y$ is the spectrum of the extension $A\into B$,  then the morphism 
	$$ {\rm Def}_X \longrightarrow \ext^1_{\cO_X}(Lf^*{\rm NL}_Y,\cO_X)$$ 
	is injective. This implies the thesis because $Y$ is the spectrum of a pointed curvilinear algebra, therefore ${\rm Def}_Y=0$.

	We can describe the morphism
	$$
	{\rm Def}_X \longrightarrow \ext^1_{\cO_X}(Lf^*{\rm NL}_Y,\cO_X)
	$$ explicitly using the Schlessinger's functors $T^i$. More precisely, it can be described as a morphism
	$$T^1(A/k,A)\rightarrow T^1(A/k,B);$$
	an easy computation shows that $T^1(A/k,A)\simeq k[t]/(t^{m-1})$ and $T^1(A/k,B)\simeq k[t]/(t^{d(m-1)})$. Through these identifications, the morphism  $$T^1(A/k,A)\rightarrow T^1(A/k,B)$$ is exactly the morphism $\phi_d$ defined earlier, i.e. it is defined by the association $t\mapsto t^dp(t)$ with $p(0)\neq 0$. The injectivity is then straightforward.
\end{proof}

Now we are ready to describe the morphism $\cE_{m,d} \rightarrow \cF_{md}$. Let us define $A_0:=\kappa[t]/(t^m)$ and $B_0:=\kappa[t]/(t^{md})$. 
Let $E_{m,d}$ be the category fibered in groupoids whose objects are $$\Big(S,(A\into B,q) ,\phi_A:(A,q)\simeq (A_0\otimes S, q_0 \otimes S), \phi_B: (B,q )\simeq (B_0 \otimes S,q_0\otimes S)\Big)$$ where $(A\into B,q) \in \cE_{m,d}(S)$. The morphism are defined in the obvious way. It is easy to see that $E_{m,d}$ is in fact fibered in sets. As before, we set $G_m:=\underaut(A_0,q_0)$ and $G_{dm}:=\underaut(B_0,q_0)$. Clearly we have a right action of $G_{dm}$ and a left action of $G_m$ on $E_{m,d}$.

\begin{proposition}
	We have the follow isomorphism of fibered categories
	$$ \cE_{m,d} \simeq [E_{m,d}/G_m\times G_{md}]$$ 
	and through this identification the morphism $\cE_{m,d} \rightarrow \cF_{dm}$  is just the projection to the classifying space $\cB G_{md}$.
\end{proposition}
\begin{proof}
	It follows from \Cref{lem:triv-ext}.
\end{proof}

\Cref{lem:triv-ext} also tells us how to describe $E_{m,d}$: the map $\phi_d$  is completely determined by $p(t) \in \cO_S[t]/(t^{d(m-1)})$ with $p(0)\in \cO_S^{\times}$. Therefore, we have a morphism 
$$(\AA^1\setminus 0)\times \AA^{d(m-1)-1} \longrightarrow E_{m,d}$$ 
which is easy to see that it is an isomorphism. Consider now the subscheme of $E_{m,d}$ defined as $$V:=\{ f \in (\AA^1\setminus 0)\times \AA^{d(m-1)-1}|\, a_0=1, \, a_{kd}(f)=0\quad {\rm for }\quad k=1,\dots,m-2\}$$
where $a_{l}(f)$ is the coefficient of $t^{l}$ of the polynomial $f(t)$. Clearly $V$ is an affine space of dimension $(m-1)(d-1)$.

\begin{lemma}\label{lem:descr-affine-bundle}
	In the situation above, we have the isomorphism
	$$ V \simeq [E_{m,d}/G_m],$$
	in particular the morphism $\cE_{m,d} \rightarrow \cF_{md}$ is an affine bundle of dimension $(m-1)(d-1)$.
\end{lemma}

\begin{proof}
	Consider the morphism 
	$$ G_m \times V \longrightarrow E_{m,d}$$
	which is just the restriction of the action of $G_m$ to $V$. A straightforward computation shows that it is an isomorphism. The statement follows.
\end{proof}

\subsection{Pushout and blowups in families}

In this subsection, we discussion two well-known constructions: pushout and blowup. Specifically, we study these two constructions in families and give conditions to get flatness and compatibility with base change. Moreover, we study when the two constructions are one the inverse of the other.

\begin{lemma}\label{lem:pushout}
	Let $S$ be a scheme. Consider three schemes $X$,$Y$ and $Y'$ which are flat over $S$. Suppose we are given $Y\hookrightarrow X$ a closed immersion and $Y\rightarrow Y'$ a finite flat morphism. Then the pushout $Y'\bigsqcup_Y X$ exists in the category of schemes, it is flat over $S$ and it commutes with base change. 
	
	Furthermore, if $X$ and $Y'$ are proper and finitely presented scheme over $S$, the same is true for $Y' \bigsqcup_Y X$.
\end{lemma}

\begin{proof}
	The existence of the pushout follows from Proposition 37.65.3 of \cite{StProj}. In fact, the proposition tells us that the morphism $Y'\rightarrow Y'\bigsqcup_Y X$ is a closed immersion and $X \rightarrow Y'\bigsqcup_Y X$ is integral, in particular affine. It is easy to prove that it is in fact finite and surjective because $Y\rightarrow Y'$ is finite and surjective.  Let us prove that $Y'\bigsqcup_Y X \rightarrow S$ is flat. Because flatness is local condition and all morphisms are affine, we can reduce to a statement of commutative algebra.
	
	Suppose we are given $R$ a commutative ring and $A$,$B$ two flat algebras. Let $I$ be an ideal of $A$ such that $A/I$ is $R$-flat and $B\hookrightarrow A/I$ be a finite flat extension. Then $B\times_{A/I}A$ is $R$-flat. To prove this, we complete the fiber square with the quotients 
	$$
	\begin{tikzcd}
		0 \arrow[r] & B\times_{A/I}A \arrow[d, two heads] \arrow[r, hook] & A \arrow[d, two heads] \arrow[r] & Q' \arrow[d] \arrow[r] & 0 \\
		0 \arrow[r] & B \arrow[r, hook]                                   & A/I \arrow[r]                    & Q \arrow[r]                                 & 0
	\end{tikzcd}
	$$
	and we notice that $Q'\rightarrow Q$ is an isomorphism. Because the extension $B\hookrightarrow A/I$ is flat, then $Q$ (thus $Q'$) is $R$-flat and the $R$-flatness of $A$ and $Q'$ implies the flatness of $B\times_{A/I}A$.  
	
	Suppose now we have a morphism $T\rightarrow S$. This induces a natural morphism 
	$$\phi_T:Y'_T \bigsqcup_{Y_T} X_T \rightarrow (Y'\bigsqcup_Y X)_T$$ where by $(-)_T$ we denote the base change $(-)\times_S T$. Because being an isomorphism is a local condition, we can reduce to the following commutative algebra's statement. Suppose we have a morphism $R\rightarrow \widetilde{R}$; we can consider the same morphism of exact sequences of $R$-modules as above and tensor it with $\widetilde{R}$. We denote by $\widetilde{(-)}$ the tensor product $(-)\otimes_R \widetilde{R}$. The flatness of $Q$ implies that the commutative diagram
	$$
	\begin{tikzcd}
		0 \arrow[r] & \widetilde{B\times_{A/I}A} \arrow[d, two heads] \arrow[r, hook] & \widetilde{A} \arrow[d, two heads] \arrow[r] & \widetilde{Q'}\arrow[d] \arrow[r] & 0 \\
		0 \arrow[r] & \widetilde{B} \arrow[r, hook]                                     & \widetilde{A/I} \arrow[r]                    & \widetilde{Q} \arrow[r]            & 0
	\end{tikzcd}
	$$
	is still a morphism of exact sequences. Because also $\widetilde{B} \times_{\widetilde{A/I}}  \widetilde{A}$ is the kernel of $\widetilde{A} \rightarrow \widetilde{Q'}$, we get that $\phi_T$ is an isomorphism. 
	
	Finally, using the fact that the pushout is compatible with base change, we reduce to the case $S=\spec R_0$ with $R_0$ of finite type over $\kappa$. Thus we can use Proposition 37.65.5 of \cite{StProj} to prove that the pushout (in the situation when $Y'\rightarrow Y$ is flat) preserves the property of being proper and finitely presented.
\end{proof}

The second construction we want to analyze is the blow-up.

\begin{lemma}\label{lem:blowup}
	Let $X/S$ be a flat, proper and finitely presented morphism of schemes and let $I$ be an ideal of $X$ such that $\cO_{X}/I^{m}$ is flat over $S$ for every $m\geq 0$. Denote by ${\rm Bl}_IX\rightarrow X$ the blowup morphism, then ${\rm Bl}_IX \rightarrow S$ is flat, proper and  finitely presented and its formations commutes with base change.
\end{lemma}

\begin{proof}
	The flatness of $\cO_X/I^m$ implies the flatness of $I^m$, therefore $\oplus_{m\geq 0}I^m$ is clearly flat over $S$. The universal property of the blowup implies that it is enough to check that the formation of the blowup commutes with base change when we restrict to the fiber of a closed point $s \in S$. Therefore it is enough to prove that the inclusion $m_sI^m\subset m_s\cap I^m$ is an equality for every $m\geq 0$ with $m_s$ the ideal associated to the closed point $s$. The lack of surjectivity of that inclusion is encoded in $\tor_S^1(\cO_X/I^m,\cO_S/m_S)$ which is trivial due to the $S$-flatness of $\cO_X/I^m$. The rest follows from classical results.
\end{proof}
\begin{remark}
	Notice that the flatness of the blowup follows from the flatness of $I^m$ for every $m$ whereas we need the flatness of $\cO_X/I^m$ to have the compatibility with base change.
\end{remark}
\begin{lemma}\label{lem:cond-diag}
	Let $R$ be a ring and $A\into B$ be an extension of $R$-algebras. Suppose $I\subset A$ is an ideal of $A$ such that $I=IB$. Then the following cocartesian diagram
	$$
	\begin{tikzcd}
		A \arrow[r, hook] \arrow[d, two heads] & B \arrow[d, two heads] \\
		A/I \arrow[r, hook]                    & B/I                   
	\end{tikzcd}
	$$
	is a cartesian diagram of $R$-algebras. Furthermore, suppose we have a cartesian diagram of $R$-algebras
	$$
	\begin{tikzcd}
		A:=\widetilde{A}\times_{B/I}B \arrow[d] \arrow[r, hook] & B \arrow[d] \\
		\widetilde{A} \arrow[r, hook]                       & B/I                      
	\end{tikzcd}
	$$
	then the morphism $A\rightarrow \widetilde{A}$ is surjective, its kernel coincides (as an $R$-module) with the ideal $I$ and the diagram is cocartesian.
\end{lemma}

\begin{proof}
	It follows from a straightforward computation in commutative algebra.
\end{proof}	

Finally, we prove that the two constructions are one the inverse of the other.

\begin{proposition}\label{prop:pushout-blowup}
	Suppose we are given a diagram
	$$
	\begin{tikzcd}
		\widetilde{D} \arrow[d] \arrow[r, hook] & \widetilde{X} \\
		D                                       &              
	\end{tikzcd}
	$$ of proper, flat, finitely presented schemes over $S$ such that $\widetilde{D} \rightarrow D$ is a finite flat morphism and $\widetilde{D}\into \widetilde{X}$ is a closed immersion of an effective Cartier divisor. Consider the pushout $X:=\widetilde{X}\amalg_{\widetilde{D}}D$ as in \Cref{lem:pushout} and denote by $I_D$ the ideal associated with the closed immersion $D\into X$. Then the pair $(X,I_D)$ over $S$ verifies the hypothesis of \Cref{lem:blowup}. Furthermore, if we denote by $(\overline{X},\overline{D})$ the blowup of the pair $(X,D)$, there exists a unique isomorphism $(\widetilde{X},\widetilde{D}) \simeq (\overline{X},\overline{D})$ of pairs over $(X,D)$.
\end{proposition}

\begin{proof}
	Consider the pushout diagram over $S$
	$$
	\begin{tikzcd}
		\widetilde{D} \arrow[d] \arrow[r, hook] & \widetilde{X} \arrow[d] \\
		D \arrow[r, hook]                       & X;                     
	\end{tikzcd}
	$$
	because every morphism is finite and flatness is a local condition, we can restrict ourself to the affine case and \Cref{lem:cond-diag} assures us that $I_D=I_{\widetilde{D}}$, and in particular $I_D^n=I_{\widetilde{D}}^n$ for every $n\geq 1$. Because $I_{\widetilde{D}}$ is a flat Cartier divisor, the same is true for its powers.
	
	Regarding the second part of the statement, we know that the unicity and existence of the morphism $$(\overline{X},\overline{D}) \longrightarrow (\widetilde{X},\widetilde{D})$$ 
	are assured by the universal property of the blowup. As being an isomorphism is a local property, we can reduce again to the affine case (all the morphisms involved are finite). The statement follows from the following remark.
\end{proof}

\begin{remark}
    Let $A\into B$ be an extension of algebras and $I$ be an ideal of $A$ such that $I=IB$ and $I$ is free of rank $1$ as a $B$-module. Therefore we can describe the Rees algebra as follows
	$$R_A(I):=\bigoplus_{n\geq 0}I^n=A\oplus tB[t]$$
	because $I$ is free of rank $1$ over $B$. It is immediate to see that the morphism $\spec B \rightarrow \proj_A(R_A(I))$ is an isomorphism over $\spec A$.
\end{remark}

\begin{corollary}
  Let $b:X\rightarrow Y$ be a finite birational morphism and let $J_b$ be the conductor ideal. Suppose that $J_b$ is a line bundle as a $\cO_X$-module. If we denote by $J\into Y$ the closed subscheme associated to $J_b$, then there exists a canonical isomorphism 
   $$ X \simeq Bl_J Y$$
   over $Y$.
\end{corollary}

\begin{proof}
 It follows from the previous remark.
\end{proof}

\begin{proposition}\label{prop:blowup-pushout}
	Let $D\into X$ be a closed immersion of proper, flat, finitely presented schemes over $S$ such that the ideal $I_D^n$ is $S$-flat for every $n\geq 1$, consider the blowup $b:\widetilde{X}:={\rm Bl}_DX\rightarrow X$ and denote by $\widetilde{D}$ the proper transform of $D$. Suppose that $\widetilde{D}\rightarrow D$ is finite flat (in particular the morphism $b$ is finite birational) and that $b^{-1}I_D = I_{\widetilde{D}}$. Then it exists a unique isomorphism  $\widetilde{X}\amalg_{\widetilde{D}} D\rightarrow X$, which makes everything commutes.  
\end{proposition}

\begin{proof}
	As in the previous proposition, the existence and unicity of the morphism are conseguences of the universal property of the pushout. Therefore as all morphisms are finite we can restrict to the affine case. We can therefore use \Cref{lem:cond-diag} and conclude.
\end{proof}

\section{The substack of $A_r$-singularity}\label{sec:4}

In this section we describe the closed substack of $\Mtilde_g^r$ which parametrizes $A_r$-stable curves with at least a singularity of type $A_h$ with $h\geq 2$. We do so by stratifying this closed substack considering singularities of type $A_h$ with a fixed $h$ greater or equal than $2$. We want to remark that all the following proofs can be upgraded to the pointed case with close to zero effort. However, we decided to treat the non-pointed case for the sake of exposition. From now on, $n$ is a positive integer unrelated with the number of sections.

Let $g\geq 2$ and $r\geq 1$ be two integers and $\kappa$ be the base field of characteristic either $0$ or greater than $r$. We recall the sequence of open subset (see \Cref{rem: max-sing})
$$ \Mtilde_g^0 \subset \Mtilde_g^1 \subset \dots \subset \Mtilde_g^r$$ 
and we define $\widetilde{\cA}_{\geq n}:=\Mtilde_g^r\setminus \Mtilde_g^{n-1}$ for $n=0,\dots,r+1$ setting $\Mtilde_g^{-1}:=\emptyset$.

We now introduce an alternative to $\widetilde{\cA}_{\geq n}$ which is easier to describe. Suppose $n$ is a positive integer smaller or equal than $r$ and let $\cA_{\geq n}$ be the substack of the universal curve $\Ctilde_g^r$ of $\Mtilde_g^r$ parametrizing pairs $(C/S,p)$ where $p$ is a section whose geometric fibers over $S$ are $A_r$-singularities for $r\geq n$. We give to $\cA_{\geq n}$ the structure of closed substack of $\Ctilde_g^r$ inductively on $n$. Clearly if $n=0$ we have $\cA_{\geq 0}=\Ctilde_g^r$. To define $\cA_{\geq 1}$, we need to find the stack-theoretic structure of the singular locus of the natural morphism $\Ctilde_g^r \rightarrow \Mtilde_g^r$. This is standard and it can be done by taking the zero locus of the $1$-st Fitting ideal of $\Omega_{\Ctilde_g^r|\Mtilde_g^r}$. We have that $\cA_{\geq 1}\rightarrow  \Mtilde_g^r$ is now finite and it is unramified over the nodes, while it ramifies over the more complicated singularities. Therefore, we can denote by $\cA_{\geq 2}$ the substack of $\cA_{\geq 1}$ defined by the $0$-th Fitting ideal of $\Omega_{\cA_{\geq 1}|\Mtilde_g^r}$. A local computation shows us that $\cA_{\geq 2} \rightarrow \Mtilde_g^r$ is unramified over the locus of $A_2$-singularities and ramified elsewhere. Inductively, we can iterate this procedure considering the $0$-th Fitting ideal of $\Omega_{\cA_{\geq n-1}|\Mtilde_g^r}$ to define $\cA_{\geq n}$. 

A local computation shows that the geometric points of $\cA_{\geq n}$ are exactly the pairs $(C,p)$ such that $p$ is an $A_{n'}$-singularity for $n\leq n'\leq r$.

Let us define $\cA_n:=\cA_{\geq n}\setminus \cA_{\geq n+1}$ for  $n=0,\dots,r-1$. We have a stratification of $\cA_{\geq 2}$
$$ \cA_{r}=\cA_{\geq r} \subset \cA_{\geq r-1} \subset \dots \subset \cA_{\geq 2}$$ 
where the $\cA_n$'s are the associated locally closed strata for $n=2,\dots, r$. 

The first reason we choose to work with $\cA_{\geq n}$ instead of $\widetilde{\cA}_{\geq n}$ is the smoothness of the locally closed substack $\cA_n$ of $\Ctilde_g^r$.

\begin{proposition}
	The stack $\cA_n$ is smooth.
\end{proposition}

\begin{proof}
	We can adapt the proof of Proposition 1.6 of \cite{DiLorPerVis} perfectly. The only thing to point out is that the \'etale model induced by the deformation theory of the pair $(C,p)$  would be $y^2=x^n+a_{n-2}x^{n-2}+\dots+a_1x+a_0$, thus the restriction to $\cA_n$ is described by the equation $a_{n-2}=\dots=a_1=a_0$. The smoothness of $\Mtilde_g^r$ implies the statement.
\end{proof}

Before going into details for the odd and even case, we describe a way of desingularize a $A_n$-singularity.

\begin{lemma}\label{lem:blowup-an}
	Let $(C,p) \in \cA_n(S)$, then the ideal $I_p$ associated to the section $p$ verifies the hypothesis of \Cref{lem:blowup}. If we denote by $b:\widetilde{C}\rightarrow C$ the blowup morphism and by $D$ the preimage $b^{-1}(p)$, then $D$ is finite flat of degree $2$ over $S$.
	\begin{itemize}
		\item If $n=1$, $D$ is a Cartier divisor of $\widetilde{C}$ \'etale of degree $2$ over $S$.	
		\item If $n\geq 2$, the $0$-th Fitting ideal of $\Omega_{D|S}$ define a section $q$ of $D\subset \widetilde{C}\rightarrow S$ such that $\widetilde{C}$ is an $A_r$-prestable curve and $q$ is an $A_{n-2}$-singularity of $\widetilde{C}$. 
	\end{itemize} 
\end{lemma}

\begin{proof}
	Suppose that the statement is true when $S$ is reduced. Because $\cA_n$ is smooth, we know that up to an \'etale cover of $S$, our object $(C,p)$ is the pullback of an object over a smooth scheme, therefore reduced. All the properties in \Cref{lem:blowup-an} are stable by base change and satisfies \'etale descent, therefore we have the statement for any $S$. 
	
	Assume that $S$ is reduced. To prove the first part of the statement, it is enough to prove that the geometric fiber $\cO_{X_s}/I_s^m$ has constant length over every point $s \in S$. This follows from a computation with the complete ring in $p_s$. Regarding the rest of the statement, we can restrict to the geometric fibers over $S$ and reduce to the case $S=\spec k$ where $k$ is an algebraically closed field over $\kappa$. The statement follows from a standard blowup computation.
\end{proof}

Let us start with $(C,p)\in \cA_n(S)$. We can construct a (finite) series of blowups which desingularize the family $C$ in the section $p$. 

Suppose $n$ is even. If we apply \Cref{lem:blowup-an} iteratively, we get the successive sequence of blowups over $S$
$$ 
\begin{tikzcd}
	\widetilde{C}_m \arrow[r, "b_m"] & \widetilde{C}_{m-1} \arrow[r, "b_{m-1}"] & \dots \arrow[r, "b_1"] & \widetilde{C}_0:=C
\end{tikzcd}
$$
with sections $q_h:S \rightarrow \widetilde{C}_h$ where $m:=n/2$, the morphism $b_h:\widetilde{C}_m:={\rm Bl}_{q_{h-1}}\widetilde{C}_{h-1}\rightarrow \widetilde{C}_{h-1}$ is the blowup of $\widetilde{C}_{h-1}$ with center $q_{h-1}$ ($q_0:=p$) and $q_h$ is the section of $\widetilde{C}_h$ over $q_{h-1}$ as in \Cref{lem:blowup-an}. We have that $\widetilde{C}_m$ is an $A_r$-prestable curve of genus $g-m$ and $q_m$ is a smooth section. 

On the contrary, if $n:=2m-1$ is odd, the same sequence of blowups gives us a curve $\widetilde{C}_m$  which has arithmetic genus either $g-m$ or $g-m+1$ depending on whether the geometric fibers of $\widetilde{C}_m$ are connected or not and an  \'etale Cartier divisor $D$ of degree $2$ over $S$.

\begin{definition}
	Let $(C,p)$ be an object of $\cA_n(S)$ with $n=2m$ or $n=2m-1$ for $S$ any scheme. The composition of blowups $b_m\circ b_{m-1} \circ \dots \circ b_{1}$ described above is called the relative $A_n$-desingularization and it is denoted as $b_{C,p}$. By abuse of notation, we refer to the source of the relative $A_n$-desingularization as relative $A_n$-desingularization. We also denote by $J_b$ the conductor ideal associated to it. 
	
	We say that an object $(C/S,p)$ of $\cA_n$ is a separating $A_n$-singularity if the geometric fibers over $S$ of the relative $A_n$-desingularization are not connected. 
\end{definition}
\begin{remark}
	By construction, the relative $A_n$-desingularization is compatible with base change. In fact, it can be described as a single blowup of $C$ in the closed subscheme defined by the conductor ideal of the desingularization.
\end{remark}

\begin{lemma}\label{lem:conductor}
	Let $(C,p)$ be an object of $\cA_n(S)$ with $n=2m$ or $n=2m-1$ and let $b_{C,p}:\widetilde{C}\rightarrow C$ the relative $A_n$-desingularization. We have that $J_b$ is flat over $S$ and its formation is compatible with base change over $S$. Furthermore, we have that $J_b=I_{b^{-1}(p)}^m$ as an ideal of $\widetilde{C}$, where $I_{b^{-1}(p)}$ is the ideal associated to the preimage of $p$ through $b$.
\end{lemma}

\begin{proof}
	The flatness and compatibility with base change are standard. If $S$ is the spectrum of an algebraically closed field, we know that the equality holds. We can consider the diagram
	$$
	\begin{tikzcd}
		&                       & f_*I_{b^{-1}(p)}^m \arrow[d]                       &             &   \\
		0 \arrow[r] & \cO_C \arrow[r, hook] & f_*\cO_{\widetilde{C}} \arrow[r, two heads] & Q \arrow[r] & 0
	\end{tikzcd}
	$$ 
	and show that the composite morphism $f_*I_{b^{-1}(p)}^m\rightarrow Q$ is the zero map, restricting to the geometric fibers over $S$ (in fact they are both finite and flat over $S$). Therefore $I_{b^{-1}(p)}^m$ can be seen also as an ideal of $\cO_{C}$. Because the conductor ideal is the largest ideal of $\cO_{C}$ which is also an ideal of $\cO_{\widetilde{C}}$, we get an inclusion $I_{b^{-1}(p)}^m\subset J_f$ whose surjectivity can be checked on the geometric fibers over $S$.
\end{proof}

\begin{remark}\label{rem:stab}
	The stability condition for $\widetilde{C}$ can be described using the Noether formula (see Proposition 1.2 of \cite{Cat}). We have that $\omega_{C/S}$ is ample if and only if $\omega_{\widetilde{C}/S}(J_b^{\vee})$ is ample.
\end{remark}

Lastly, we prove that the stack parametrizing separating $A_n$-singularities for a fixed positive integer $n$ is closed inside $\Ctilde_g^r$. 

\begin{lemma}\label{lem:sep-sing}
	Let $C\rightarrow \spec R$ a family of $A_r$-prestable curves over a DVR and denote by $K$ the function field of $R$. Suppose there exists a generic section $s_K$ of the morphism such that its image is a separating $A_{r_0}$-singularity (with $r_0\leq r$). Then the section $s_R$ (which is the closure of $s_K$) is still a separating $A_{r_0}$-singularity. 
\end{lemma}

\begin{proof}
	Because $s_K$ is a separating $A_{r_0}$-singularity, then $r_0$ is necessarily odd. Furthemore, we have that the special fiber $s_k:=s_R\otimes_R k$ is an $A_{r_1}$-singularity with $r_1\geq r_0$.
	
	Let us call $m_R$ the ideal associated with the section $s_R$. Because $C/\spec R$ is $A_r$-prestable, we can compute $\dim_L(\cO_C/I_R^h\otimes_R L)$ when $L$ is the algebraic closure of either the function field $K$ or the residue field $k$. We have that 
	$$ \dim_L L[[x,y]]/(y^2-x^n,m^h)=2h-1$$ 
	for every $h\geq 1$ and every $n\geq 2$, where $m=(x,y)$. 
	
	Therefore, $\dim_L(\cO_C/I_R^h\otimes_R L)$ is constant on the geometric fibers over $\spec R$ and we get that $\cO_C/I_R^h$ is $R$-flat  because $R$ is reduced.  Consider now $\widetilde{C}_1:={\rm Bl}_{I_R}C$, which is still $R$-flat, proper and finitely presented thanks to \Cref{lem:blowup} and commutes with base change. We denote by $b:\widetilde{C}_1 \rightarrow C$ the blowup morphism. We have that a local computations shows that if $r_0\geq 3$ we have $b^{-1}(s_R)_{\rm red}=\spec R$ and thus it defines a section $q_1$ of $\widetilde{C}_1$ which is a $A_{r_0-2}$-singularity at the generic fiber and a $A_{r_1-2}$-singularity at the special fiber. We can therefore iterate this procedure $r_0/2$ times until we get $\widetilde{C}_{r_0/2}\rightarrow \spec R$ which is a flat proper finitely presented morphism whose generic fiber is not geometrically connected. Therefore the special fiber is not geometrically connected as it is geometrically reduced. This clearly implies that $r_1=r_0$ and that $s_R$ is a separating section.
\end{proof}

\subsection*{Description of $\cA_n$ for the even case}

Let $n:=2m$ be an even number with $m\geq 1$. Firstly, we study the $A_{2m}$-singularity locally, and then we try to describe everything in families of projective curves.

Let $(C,p)$ be a $1$-dimensional reduced $1$-pointed scheme of finite type over an algebraically closed field where $p$ is an $A_{2m}$-singularity. Consider now the partial normalization in $p$ which gives us a finite birational morphism 
$$b:\widetilde{C} \longrightarrow C$$ 
which is infact an homeomorphism. This implies that the only think we need to understand is how the structural sheaf changes through $b$. We have the standard exact sequence
$$ 0 \rightarrow \cO_C \rightarrow \cO_{\widetilde{C}} \rightarrow Q \rightarrow 0$$
where $Q$ is a coherent sheaf on $C$ with support on the point $p$. Consider now the conductor ideal $J_b$ of the morphism $b$, which is both an ideal of $\cO_C$ and of $\cO_{\widetilde{C}}$.  Consider the morphism of exact sequences
$$
\begin{tikzcd}
	0 \arrow[r] & \cO_C \arrow[r] \arrow[d, two heads] & \cO_{\widetilde{C}} \arrow[r] \arrow[d, two heads] & Q \arrow[r] \arrow[d, Rightarrow, no head] & 0 \\
	0 \arrow[r] & \cO_C/J_b \arrow[r]                    & \cO_{\widetilde{C}}/J_b \arrow[r]                    & Q \arrow[r]                                & 0,
\end{tikzcd}
$$
it is easy to see that the vertical morphism on the right is an isomorphism. Moreover, a local computation shows that $J_b=m_q^{2m}$ as an ideal of $\cO_{\widetilde{C}}$. Therefore \Cref{lem:cond-diag} and \Cref{lem:conductor} imply that to construct an $A_{2m}$-singularity we need the partial normalization $\widetilde{C}$, the section $q$ and a subalgebra of $\cO_{\widetilde{C}}/m_q^{2m}$. Notice that not every subalgebra works. 

\begin{remark}
	First of all, a local computation shows that the extension $\cO_{C}/J_b \into \cO_{\widetilde{C}}/m_q^{2m}$ is finite flat of degree $2$. Luckily, the converse is also true: if $B=k[[t]]$ and $I=(t^{2m})$, then the subalgebras $C$ such that $C \into B/I$ is finite flat of degree $2$ are exactly the ones whose pullback through the projection $B \rightarrow B/I$ is an $A_{2m}$-singularity. This should serve as a motivation for the alternative description we are going to prove for $\cA_{2m}$.
\end{remark}  

The idea now is to prove that the same exact picture works for families of curves. The result holds in a greater generality, but we describe directly the case of families of $A_r$-prestable curves. 

We want to construct an algebraic stack whose objects are triplets $(\widetilde{C}/S,q,A)$ where $(\widetilde{C}/S,q)$ is a $A_r$-prestable $1$-pointed curve of genus $g-m$ with some stability condition (see \Cref{rem:stab}) and $A\subset \cO_{\widetilde{C}}/I_q^{2m}$ is a finite flat extension  of degree $2$ of flat $\cO_S$-algebras, where $I_q$ is the ideal sheaf associated to the section $q$. 

Firstly, we introduce the stack $\Mtilde_{h,[l]}^r$ parametrizing $A_r$-prestable 1-pointed curves $(\widetilde{C},q)$ such that $\omega_{\widetilde{C}}(lq)$ is ample. This is not difficult to describe, in fact we have a natural inclusion
$$ \Mtilde_{h,[l+1]}^r\subset \Mtilde_{h,[l]}^r$$ 
which is an equality for $l\geq 2$ if $h\geq 1$ and for $l\geq 3$ if $h=0$. We have that the only curves that live in $\Mtilde_{h,[l]}^r\setminus \Mtilde_{h,1}$ are curves that have one (irreducible) tail of genus $0$ and the section lands on the tail. We have the following result.

\begin{proposition}\label{prop:desc-str}
	In the situation above, if $h\geq 1$ we have an isomorphism 
	$$ \Mtilde_{h,[l]}^r \simeq \Mtilde_{h,1}^r\times [\AA^1/\gm]$$ 
	for $l\geq 2$. If $h=0$, we have an isomorphism 
	$$ \Mtilde_{0,[l]}\simeq \cB(\gm \rtimes \ga)$$ 
	for $l\geq 3$.
\end{proposition} 

\begin{proof}
	The case $h=0$ is straightforward. Clearly for $h\geq 1$ it is enough to construct the isomorphism for $l=2$.
	
	In the proof of Theorem 2.9 of \cite{DiLorPerVis}, they proved the description for $r=2$, but the proof can be generalized easily for any $r$. We give an alternative proof. First of all, it is easy to see that $\Mtilde_{h,[l]}^r$ is smooth using deformation theory.
	
	We follow the construction introduced in Section 2.3 of \cite{DiLorPerVis}: consider $\Ctilde_{h,1}^r$ the universal curve of $\Mtilde_{h,1}^r$ and define $\cD_{h,1}$ to be the blowup of $\Ctilde_{h,1}^r \times [\AA^1/\gm]$ in the center $\Mtilde_{h,1}^r\times \cB\gm$, where $\Mtilde_{h,1}^r\hookrightarrow \Ctilde_{h,1}^r$ is the universal section. If we denote by $q$ the proper transform of the closed substack
	$$\Mtilde_{h,1}^r \times [\AA^1/\gm] \into \Ctilde_{h,1}^r \times [\AA^1/\gm]$$
	we get that $(\cD_{h,1}\rightarrow \Mtilde_{h,1}^r \times [\AA^1/\gm],q)$ define a morphism 
	$$ \varphi:\Mtilde_{h,1}^r \times [\AA^1/\gm] \longrightarrow \Mtilde_{h,[l]}^r$$
	which by construction is birational, as it is an isomorphism restricted to $\Mtilde_{h,1}^r$. 
	
	We define a morphism $\phi:=(\phi_1,\phi_2):\Mtilde_{h,[l]}^r \rightarrow \Mtilde_{h,1}^r \times [\AA^1/\gm]$ and prove that it is a quasi-inverse of $\varphi$.
	
	Let $(\tilde{C},q)$ be an object of $\Mtilde_{h,[l]}^r$ and let $\Sigma$ be the Cartier divisor $\cO(q)$ associated to the section $q$. Then $(\tilde{C},\Sigma)$ is a $A_r$-semistable pair, therefore we can apply \Cref{prop:unique-contrac} to get the contraction $c:\tilde{C} \rightarrow C$. We define the image of $(\tilde{C},q)$ through $\phi_1$ to be the $1$-pointed $A_r$-stable curve $(C,p)$ where $p:=c(q)$.     
	
	We denote by $\Delta_0$ the substack of $\Mtilde_{h,[l]}^r$ parametrizing curves with a separating node such that one of the two connected components of the normalization is of genus $0$. This is closed because of \Cref{lem:sep-sing}. It is easy to see that it is in fact a Cartier divisor inside $\Mtilde_{h,[l]}^r$. Therefore we can define $\phi_2:\Mtilde_{h,[l]}^r \rightarrow [\AA^1/\gm]$ as the morphism associated to the effective Cartier divisor $\Delta_0$. 
	
	\Cref{prop:unique-contrac} implies that there exists an isomorphism of functors $\phi \circ \varphi \simeq \id$. The universal property of the blowup implies that we have a natural transformation $\id \Rightarrow \varphi \circ \phi$ and we can check that it is an isomorphism on the geometric points of $\Mtilde_{h,1}^r \times [\AA^1/\gm]$. \Cref{lem:contrac-unique} concludes the proof. 
\end{proof}

Finally, we are ready to describe $\cA_n$. In \Cref{sub:3.1}, we define the stack $\cF_n^c$ which parametrizes pointed finite flat curvilinear algebras of length $n$ and $\cE_{m,d}^c$ which parametrizes finite flat curvilinear extensions of degree $d$ of pointed finite flat curvilinear algebras of degree $m$. We proved that the natural morphism 
$$ \cE_{m,d}^c \longrightarrow \cF_{md}^c$$
defined by the association $(A\into B) \mapsto B$ is an affine bundle. See \Cref{lem:descr-affine-bundle}. 

Let 
$$ \cE_{m,2}^c\longrightarrow \cF_{2m}^c$$
as above and consider the morphism of stacks
$$ \Mtilde_{g-m,[2m]}^r\longrightarrow \cF_{2m}^c$$ 
defined by the association $(\widetilde{C},q) \mapsto \cO_{\widetilde{C}}/I_q^{2m}$ where $I_q$ is the ideal defined by the section $q$. 

We denote by $\cA'_{2m}$ the fiber product $\cE_{m,2}^c \times_{\cF_{2m}^c} \Mtilde_{g-m,[2m]}^r$. By definition, an object of $\cA_{2m}'$ over $S$ is of the form $(\widetilde{C},q, A \subset \cO_{\widetilde{C}}/I_q^{2m})$ where $A \subset \cO_{\widetilde{C}}/I_q^{2m}$ is a finite flat extension of algebras of degree $2$. Given two objects $(\pi_1:\widetilde{C}_1/S,q_1, A_1 \subset \cO_{\widetilde{C}_1}/I_{q_1}^2m)$ and $(\pi_2:\widetilde{C}_2/S,q_2, A_2 \subset \cO_{\widetilde{C}}/I_{q_2}^2m)$, a morphism over $S$ between them is a pair $(f,\alpha)$ where $f:(\widetilde{C}_1,q_1)\rightarrow (\widetilde{C}_2,q_2)$  is a morphism in $\Mtilde_{g-m,[2m]}^r(S)$ while $\alpha:A_2 \rightarrow A_1$ is an isomorphism of finite flat algebras over $S$ such that the diagram 
$$
\begin{tikzcd}
	A_2 \arrow[r, "\alpha"] \arrow[d, hook]            & A_1 \arrow[d, hook]                        \\
	\pi_{2,*}(\cO_{\widetilde{C}_2}/I_{q_2}^{2m}) \arrow[r] & \pi_{1,*}(\cO_{\widetilde{C}_1}/I_{q_1}^{2m})
\end{tikzcd}
$$
is commutative. 

We want to construct a morphism from $\cA'_{2m}$ to $\cA_{2m}$. Let $S$ be a scheme and $(\widetilde{C},q, A \subset \cO_{\widetilde{C}}/I_q^{2m})$ be an object of $\cA'_{2m}(S)$. We consider the diagram 
$$
\begin{tikzcd}
	\spec_{\cO_S}(\cO_{\widetilde{C}}/I_q^{2m}) \arrow[d, "2:1"] \arrow[r, hook] & \widetilde{C} \arrow[dd, bend left] \\
	\spec_{\cO_S}(A) \arrow[rd]                                                  &                                     \\
	& S                                  
\end{tikzcd}
$$
and complete it with the pushout (see \Cref{lem:pushout}), which has a morphism over $S$:
$$
\begin{tikzcd}
	\spec_{\cO_S}(\cO_{\widetilde{C}}/I_q^{2m}) \arrow[d, "2:1"] \arrow[r, hook] & \widetilde{C} \arrow[dd, bend left] \arrow[d, dotted] \\
	\spec_{\cO_S}(A) \arrow[rd] \arrow[r, dotted]                                & C \arrow[d, dotted]                                   \\
	& S.                                                    
\end{tikzcd}
$$
In this case, $\widetilde{C}$ and $C$ share the same topological space, whereas the structural sheaf $\cO_C$ of the pushout is the fiber product 
$$
\begin{tikzcd}
	\cO_C:=\cO_{\widetilde{C}}\times_{\cO_{\widetilde{C}}/I_q^{2m}} A \arrow[d, two heads] \arrow[r, hook] & \cO_{\widetilde{C}} \arrow[d, two heads] \\
	A \arrow[r, hook]                                                                                 & \cO_{\widetilde{C}}/I_q^{2m};         
\end{tikzcd}
$$
we define $I_p:=I_q\vert_{\cO_C}$ which induces a section $p:S \rightarrow C$ of $C\rightarrow S$.

\begin{lemma}
	In the situation above, $C/S$ is an $A_r$-stable curve of genus $g$ and $p_s$ is a $A_{2m}$-singularity for every geometric point $s \in S$. Furthermore, the formation of the pushout commutes with base change over $S$.
\end{lemma}

\begin{proof}
	The fact that $C/S$ is flat, proper and finitely presented is a consequence of \Cref{lem:pushout}. The same is true for the compatibility with base change over $S$. We only need to check that $C_s$ is an $A_r$-stable curve for every geometric point $s \in S$. Therefore we can assume $S=\spec k$ with $k$ an algebraically closed field over $\kappa$. Connectedness is trivial as the topological space is the same. We need to prove that $p$ is an $A_{2m}$-singularity. Consider the cartesian diagram of local rings
	$$
	\begin{tikzcd}
		{\cO_{C,p}:=\cO_{\widetilde{C},q}\times_{\cO_{\widetilde{C},p}/m_q^{2m}} A} \arrow[d, two heads] \arrow[r, hook] & {\cO_{\widetilde{C},q}} \arrow[d, two heads] \\
		A \arrow[r, hook]                                                                                     & {\cO_{\widetilde{C},q}/m_q^{2m}}            
	\end{tikzcd}
	$$
	and pass to the completion with respect to $m_p$, the maximal ideal of $\cO_{C,p}$. Because the extensions are finite, we get the following cartesian diagram of rings
	$$
	\begin{tikzcd}
		B \arrow[d, two heads] \arrow[r, hook]   & {k[[t]]} \arrow[d, two heads] \\
		{k[[t]]/(t^m)} \arrow[r, "\phi_2", hook] & {k[[t]]/(t^{2m})};            
	\end{tikzcd}
	$$
	using the description of $\phi_2$ as in \Cref{lem:triv-ext}, it is easy to see that it is defined by the association $t \mapsto t^2$ up to an isomorphism of $k[[t]]/(t^{2m})$. This concludes the proof.
\end{proof}

Therefore we have constructed a morphism of algebraic stacks
$$F:\cA'_{2m} \longrightarrow \cA_{2m}$$ 
defined on objects by the association
$$(\widetilde{C},q,A\subset \cO_{\widetilde{C}}) \mapsto (\widetilde{C}\bigsqcup_{\spec (\cO_{\widetilde{C}}/I_q^{2m})} \spec A, p) $$
and on morphisms in the natural way using the universal property of the pushout. 

To construct the inverse, we use the relative $A_n$-desingularization which is compatible with base change. \Cref{lem:conductor} implies that we can define a functor $G:\cA_{2m}\rightarrow \cA_{2m}'$ on objects 
$$ (C/S,p) \mapsto (\widetilde{C},q, \cO_C/J_b \subset \cO_{\widetilde{C}}/I_q^{2m}) $$
where $b:\widetilde{C}\rightarrow C$ is the relative $A_{2m}$-desingularization, $J_b$ is the conductor ideal relative to $b$ and $q$ is the smooth section of $\widetilde{C}\rightarrow S$ defined as the vanishing locus of the $0$-th Fitting ideal of $\Omega_{b^{-1}(p)|S}$. It is defined on morphisms in the obvious way. 

\begin{proposition}\label{prop:descr-an-pari}
	The two morphisms $F$ and $G$ are quasi-inverse of each other.
\end{proposition}

\begin{proof}
	The statement is a conseguence of \Cref{prop:pushout-blowup} and \Cref{prop:blowup-pushout}.
\end{proof}

\begin{corollary}\label{cor:descr-an-pari}
	$\cA_{2m}$ is an affine bundle of dimension $(m-1)$ over the stack $\Mtilde^r_{g-m,[2m]} $ for $m\geq 1$.
\end{corollary}

\begin{proof}
	Because $\cA_{2m}'$ is constructed as the fiber product $\cE_{m,2}^c \times_{\cF_{2m}^c} \Mtilde_{g-m,[2m]}^r$, the statement follows from \Cref{lem:descr-affine-bundle}.
\end{proof}

\subsection{Description of $A_n$ for the odd case}

Let $n:=2m-1$ and $(C,p)$ be a $1$-dimensional reduced $1$-pointed scheme of finite type over an algebraically closed field and $p$ is an $A_{2m-1}$-singularity. Consider now the partial normalization in $p$ which gives us a finite birational morphism 
$$b:\widetilde{C} \longrightarrow C$$ 
which is not an homeomorphism, as we know that $b^{-1}(p)$ is a reduced divisor of $\widetilde{C}$ of length $2$. We can use \Cref{lem:cond-diag} to prove that the extension $\cO_C \hookrightarrow b_*\cO_{\widetilde{C}}$ can be constructed pulling back a subalgebra of $b_*\cO_{\widetilde{C}}/J_b$ through the quotient $b_*\cO_{\widetilde{C}}\rightarrow b_*\cO_{\widetilde{C}}/J_b$, as in the even case. We can describe the subalgebra in the following way.

Consider the divisor $b^{-1}(p)$ which is the disjoint union of two closed points, namely $q_1$ and $q_2$. Then the composition
$$
\begin{tikzcd}
	\cO_C/J_b \arrow[r] & f_*\cO_{\widetilde{C}}/J_b=\cO_{mq_1}\oplus \cO_{mq_2} \arrow[r] & \cO_{mq_i}
\end{tikzcd}
$$
is an isomorphism for $i=1,2$, where $\cO_{mq_i}$ is the structure sheaf of the support of the Cartier divisor $mq_i$ for $i=1,2$ and the second map is just the projection. Therefore the subalgebra $\cO_C/J_b$ of $f_*\cO_{\widetilde{C}}/J_b$ is determined by an isomorphism between $\cO_{mq_1}$ and $\cO_{mq_2}$. Recall that \Cref{rem:genus-count} implies that we can have two different situation: either $\widetilde{C}$ is connected and its genus is $g-m$  or it has two connected components of total genus $g-m+1$ and the two points lie in different components. 
\begin{definition}
	Let $0\leq i\leq (g-m+1)/2$. We define $\cA_{2m-1}^{i}$ to be the substack of $\cA_{2m-1}$ parametrizing $1$-pointed curves $(C/S,p)$ such that the geometric fibers over $S$ of the relative $A_{2m-1}$-desingularization are the disjoint union of two curves of genus $i$ and $g-m-i+1$.
	
	Furthermore, we define $\cA_{2m-1}^{\rm ns}$ the substack of $\cA_{2m-1}$ parametrizing curves $(C/S,p)$ such that the geometric fibers of the relative $A_{2m-1}$-desingularization are connected of genus $g-m$. 
\end{definition}

\begin{proposition}
	The algebraic stack $\cA_{2m-1}$ is the disjoint union of $\cA_{2m-1}^{\rm ns}$ and $\cA_{2m-1}^i$ for every $0 \leq i\leq  (g-m+1)/2$.
\end{proposition}

\begin{proof}
	Let $(C/S,p)$ be an object of $\cA_{2m-1}$ and consider the relative \\ $A_n$-desingularization $$b:(\widetilde{C},q)\rightarrow (C,p)$$
	over $S$. Because $\widetilde{C}\rightarrow S$ is flat, proper, finitely presented and the fibers over $S$ are geometrically reduced, we have that the number of connected components of the geometric fibers of $\widetilde{C}$ over $S$ is locally constant. See Proposition 15.5.7 of \cite{EGA}. Therefore we have that $\cA_{2m-1}^{\rm ns}$ is open and closed inside $\cA_{2m-1}$. Furthermore, we have that the objects of the complement $\cA_{2m-1}^{\rm s}$ of $\cA_{2m-1}^{\rm ns}$ are pairs $(C/S,p)$ such that the geometric fibers of $\widetilde{C}$ over $S$ have two connected components. Suppose that $S=\spec R$ with $R$ a strictly henselian ring over $\kappa$. We know that $b^{-1}(p)$ is \'etale of degree $2$ over $R$, thus it is the disjoint union of two copies of $R$. We denote by $q$ one of the two sections of $\widetilde{C}\rightarrow \spec R$ and we denote by $C_s^0$ the connected component of the fiber $\widetilde{C}_s$ which contains $q_s$ for every point $s \in \spec R$. Proposition 15.6.5 and Proposition 15.6.8 in \cite{EGA} imply that the set-theoretic union 
	$$ C_0:=\bigcup_{s \in S} C_s^0$$
	is a closed and open subscheme of $\widetilde{C}$, therefore the morphism $C_0\rightarrow S$ is still proper, flat and finitely presented. In particular, the arithmetic genus of the fibers is locally constant. It is easy to see that in particular $\cA_{2m-1}^{\rm s}$ is the disjoint union of $\cA_{2m-1}^{\rm i}$ for $0\leq i\leq (g-m+1)/2$.
\end{proof}

We start by describing $\cA_{2m-1}^{\rm ns}$. Let $\Mtilde_{h,2[l]}^r$ be a fibered category in groupoid over the category of schemes whose objects are of the form $(\widetilde{C}/S,q_1,q_2)$ where $\widetilde{C}\rightarrow S$ is a flat, proper, finitely presented morphism of scheme, $q_1$ and $q_2$ are smooth sections, $\widetilde{C}_s$ is an $A_r$-prestable curve of genus $h$ for every geometric point $s \in S$ and $\omega_{\widetilde{C}}(lq_1+lq_2)$ is relatively ample over $S$. The morphisms are defined in the obvious way. 

\begin{proposition}
	We have an isomorphism  of fibered category
	$$\Mtilde_{h,2[l]}^r\simeq  \Mtilde_{h,2}^r \times [\AA^1/\gm] \times [\AA^1/\gm]$$
	for $l\geq 2$ and $h\geq 1$. Furthermore, we have 
	$$ \Mtilde_{0,2[l]}\simeq \cB\gm \times [\AA^1/\gm]$$
	for $l\geq 3$.
\end{proposition}

\begin{proof}
	This follows easily from \Cref{prop:desc-str}. 
\end{proof}

We want to construct an algebraic stack with a morphism over $\Mtilde_{h,2[m]}^r$ whose fibers parametrize the isomorphisms between the two finite flat $S$-algebras $\cO_{mq_1}$ and $\cO_{mq_2}$.

\begin{remark}
	Recall that we have a smooth stack $\cE_{m,d}^c$ which parametrizes finite flat extensions $A\into B$ of degree $d$ with $B$ curvilinear of length $m$. If $d=1$, the stack $\cE_{m,1}^c$ parametrizes isomorphisms of finite flat algebras of length $m$. We also have a map 
	$$ \cE_{m,1}^c \longrightarrow \cF_m^c \times \cF_m^c$$ 
	defined on objects by the association $(A\into B) \mapsto (A,B)$, which is a $\gm \times \ga^{m-2}$-torsor. See \Cref{sub:3.1} for a more detailed discussion.
\end{remark}

Consider the morphism
$$\Mtilde_{g-m,2[m]}^r \longrightarrow \cF_{m}^c \times \cF_m^c$$
defined by the association 
$$(\widetilde{C},q_1,q_2)\mapsto (\cO_{\widetilde{C}}/I_{q_1}^m,\cO_{\widetilde{C}}/I_{q_2}^m);$$
and let $I_{2m-1}^{\rm ns}$ be the fiber product $\Mtilde_{h,2[m]}^r \times_{(\cF_{m}^c\times\cF_{m}^c)} \cE_{m,1}^c$. It parametrizes objects $(C,q_1,q_2,\phi)$ such that $(C,q_1,q_2) \in \Mtilde_{h,2[m]}^r$ and $\phi$ is an isomorphism between $\cO_{mq_1}$ and $\cO_{mq_2}$ as $\cO_S$-algebras which commutes with the sections.

We can construct a morphism 
$$ I_{2m-1}^{\rm ns} \longrightarrow \cA_{2m-1}^{\rm ns}$$ 
in the following way: let $(\widetilde{C},q_1,q_2,\phi) \in I_{2m-1}^{\rm ns}(S)$, then we have the diagram
$$
\begin{tikzcd}
	\spec_S(\cO_{mq_1})\bigsqcup \spec_S(\cO_{mq_2}) \arrow[d, "{(\id,\phi)}"] \arrow[r, hook] & \widetilde{C} \\
	\spec_S(\cO_{mq_1})                                                                     &              
\end{tikzcd}
$$
where the morphism $(\id,\phi)$ is \'etale of degree $2$. We denote by $C$ the pushout of the diagram and $p$ the image of $q$. We send $(\widetilde{C},q_1,q_2,\phi)$ to $(C,q)$. Notice that because both $q_1$ and $q_2$ are smooth sections, we have that $\cO_{mq_i}$ is the scheme-theoretic support of a Cartier divisor of $\widetilde{C}$, therefore it is flat for $i=1,2$. \Cref{lem:pushout} assures us that this construction is functorial and commutes with base change. 

\begin{proposition}\label{prop:descr-an-odd-ns}
	The pushout functor
	$$F^{\rm ns}: I_{2m-1}^{\rm ns} \longrightarrow \cA_{2m-1}^{\rm ns}$$ 
	is representable finite \'etale of degree $2$.
\end{proposition}

\begin{proof}
	
	It is a direct conseguence of both \Cref{prop:pushout-blowup} and \Cref{prop:blowup-pushout}. In fact  the two propositions assure us that the object $(C,p)$ is uniquely determined by the relative $A_{2m-1}$-desingularization $b:\widetilde{C}\rightarrow C$, the fiber $b^{-1}(p)$ and an automorphism of the $m$-thickening of $b^{-1}(p)$ which restricted to the geometric fibers over $S$ acts exchanging the two points of $b^{-1}(p)$. Because $b^{-1}(p)$ is finite \'etale of degree $2$, it is clear that $F^{\rm ns}$ is a finite \'etale morphism of degree $2$.
\end{proof}
Let $0\leq i\leq (g-m+1)/2$. In the same way, we can define morphisms 
$$  \Mtilde_{i,[m]}\times \Mtilde_{g-i-m+1,[m]} \longrightarrow \cF_{m}^c \times \cF_m^c$$ 
defined by the association $$\Big((\widetilde{C}_1,q_1),(\widetilde{C}_{2},q_{2})\Big) \mapsto \Big((\cO_{\widetilde{C_1}}/I_{q_1}^m,\cO_{\widetilde{C}_{2}}/I_{q_{2}}^m)\Big)$$
and we denote by $I_{2m-1}^{i}$ the fiber product $$(\Mtilde_{i,[m]}\times \Mtilde_{g-m-i+1,[m]})\times_{(\cF_{m}^c\times\cF_{m}^c)} \cE_{m,1}^c.$$

Similarly to the previous case, we can construct a functor 
$$ F^i: I_{2m-1}^i \longrightarrow \cA_{2m-i}^i$$ 
using the pushout construction. Again, we have the following result.

\begin{proposition}\label{prop:descr-an-odd-i}
	The functor $F^i$ is an isomorphism for $i \neq (g-m+1)/2$ whereas is finite \'etale of degree $2$ for $i = (g-m+1)/2$.
\end{proposition}

\begin{proof}
	The proof is exactly the same as \Cref{prop:descr-an-odd-ns}.
\end{proof}

\bibliographystyle{plain}
\bibliography{Bibliografia}

\end{document}